\documentclass{amsart} % AMS math article

\usepackage{amsmath,amssymb,amsthm} % formulas, mathematical symbols, theorems, etc.
\usepackage{graphicx} % for graphics
\usepackage{subfig}
\usepackage[arrow, matrix, curve]{xy} % for commutative diagrams
\usepackage{xcolor} % colors
\usepackage{makecell} % for linebreak in tables
\usepackage{multirow} % for multiple rows in tables
\usepackage{multicol}
\usepackage{enumerate}
\usepackage{booktabs} % for pretty tables
\usepackage{siunitx} % for decimal alignment
\usepackage{enumitem} % for resuming list numbering and more
\usepackage[percent]{overpic}
\usepackage{ulem}

\usepackage{url}
\usepackage{hyperref} % for links
\hypersetup{
	colorlinks,
	linkcolor={red!50!black},
	citecolor={blue!50!black},
	urlcolor={blue!80!black}
}  %%% This tones down the bright green hyperlink box and its other colors...

\DeclareMathOperator{\dbc}{DBC}
\DeclareMathOperator{\HF}{HF}

\DeclareMathOperator{\rk}{rk}
\newcommand{\calA}{\mathcal{A}}

\newcommand{\Z}{\mathbb{Z}}
\newcommand{\F}{\mathcal{F}}
\newcommand{\Q}{\mathbb{Q}}

\newcommand{\spin}{\ifmmode{\rm Spin}\else{${\rm spin}$\ }\fi}
\newcommand{\spinc}{\ifmmode{{\rm Spin}^c}\else{${\rm spin}^c$}\fi}

\newcommand{\bdry}{\partial}
\newcommand{\nbhd}{\mathcal{N}}

\newcommand{\cut}{\backslash}

\makeatletter
\newcommand{\Vast}{\bBigg@{2.5}} % or 4.3?   To make large symbol
\makeatother

\makeatletter
\newtheorem*{rep@theorem}{\rep@title}
\newcommand{\newreptheorem}[2]{%
	\newenvironment{rep#1}[1]{%
		\def\rep@title{#2 \ref{##1}}%
		\begin{rep@theorem}}%
		{\end{rep@theorem}}}
\makeatother
\newtheoremstyle{thm}{}{}{\itshape}{}{\bfseries}{}{ }{} %Thereom style
\newtheoremstyle{definition}{}{}{}{}{\bfseries}{}{ }{} %Definition style

\theoremstyle{thm}
\newtheorem{Theorem}{Theorem}[section]
\newtheorem{thm}[Theorem]{Theorem}
\newreptheorem{thm}{Theorem}  % for repeating a theorem
\newtheorem{lem}[Theorem]{Lemma}
\newreptheorem{lem}{Lemma}  % for repeating a lemma
\newtheorem{prop}[Theorem]{Proposition}
\newtheorem{cor}[Theorem]{Corollary}
\newtheorem*{Theorem-ohne}{Theorem}

\newtheorem{ques}[Theorem]{Question}

\newtheorem*{thm:counterexample}{Theorem~\ref{thm:counterexample}}
\newtheorem*{thm:torus_knotsQA}{Theorem~\ref{thm:torus_knotsQA}}
\newtheorem*{thm:QAsurgeries}{Theorem~\ref{thm:QAsurgeries}}
\newtheorem*{thm:asym_QA}{Theorem~\ref{thm:asym_QA}}
\newtheorem*{thm:exceptionalsym}{Theorem~\ref{thm:exceptionalsym}}
\newtheorem*{thm:2uncertain}{Theorem~\ref{thm:2uncertain}}
\newtheorem*{thm:BLknots}{Theorem~\ref{thm:BLknots}}
\newtheorem*{cor:formal_torus}{Corollary~\ref{cor:formal_torus}}
\newtheorem*{cor:asymmetric_formal_L_spaces}{Corollary~\ref{cor:asymmetric_formal_L_spaces}}
\newtheorem*{thm:plusone}{Theorem~\ref{thm:plusone}}
\newtheorem*{thm:formal_integer}{Theorem~\ref{thm:formal_integer}}
\newtheorem*{prop:formal_bound}{Proposition~\ref{prop:formal_bound}}

\theoremstyle{definition}
\newtheorem{defi}[Theorem]{Definition}
\newtheorem{rem}[Theorem]{Remark}
\newtheorem{ex}[Theorem]{Example}

\newtheorem{strat}[Theorem]{Strategy}

\begingroup\expandafter\expandafter\expandafter\endgroup
\expandafter\ifx\csname pdfsuppresswarningpagegroup\endcsname\relax
\else
\fi

\definecolor{amaranth}{rgb}{0.9, 0.17, 0.31} %dark red
\definecolor{carrotorange}{rgb}{0.93, 0.57, 0.13} %orange
\definecolor{citrine}{rgb}{0.89, 0.82, 0.04} %dark yellow
\definecolor{dartmouthgreen}{rgb}{0.05, 0.5, 0.06} %green
\definecolor{ballblue}{rgb}{0.13, 0.67, 0.8} %blue
\definecolor{ceruleanblue}{rgb}{0.16, 0.32, 0.75} %deeper blue
\definecolor{amethyst}{rgb}{0.6, 0.4, 0.8} %purple
\definecolor{amber}{rgb}{1.0, 0.75, 0.0} %amber
\definecolor{burlywood}{rgb}{0.87, 0.72, 0.53} %beigebrown
\numberwithin{equation}{section}
\begin{document}
	
	%%%%%%%%%%%%%%%%%%%%%%%%%%%%% Title and authors %%%%%%%%%%%%%%%%%%%%%%%%%%%%%%%%%%%%
	
	\title{Quasi-alternating surgeries} 
	
	\author{Kenneth L. Baker}
	\address{Department of Mathematics, University of Miami, Coral Gables, FL 33146, USA}
	\email{k.baker@math.miami.edu}
	
\author{Marc Kegel}
\address{Universidad de Sevilla, Dpto.\ de Álgebra,
Avda.\ Reina Mercedes s/n,
41012 Sevilla}
\email{mkegel@us.es, kegelmarc87@gmail.com}
	
	\author{Duncan McCoy}
	\address{D\'{e}partment de Math\'{e}matiques, Universit\'{e} du Qu\'{e}bec \`{a} Montr\'{e}al, Canada}
	\email{mc\_coy.duncan@uqam.ca}
	
	%%%%%%%%%%%%%%%%%%%%%%%%%%%%% Abstract %%%%%%%%%%%%%%%%%%%%%%%%%%%%%%%%%%%%
	
	\date{\today} % date on first page
	
\begin{abstract}
	In this article, we explore phenomena relating to quasi-alternating surgeries on knots, where a quasi-alternating surgery on a knot is a Dehn surgery yielding the double branched cover of a quasi-alternating link. Since the double branched cover of a quasi-alternating link is an L-space, quasi-alternating surgeries are special examples of L-space surgeries.
		
	We show that all SnapPy census L-space knots admit quasi-alternating surgeries except for the knots $t09847$ and $o9\_30634$, neither of which have any quasi-alternating surgeries. In particular, this finishes Dunfield's classification of the L-space knots among all SnapPy census knots. In addition, we show that all asymmetric census L-space knots have exactly two quasi-alternating slopes and that these are consecutive integers. Similar behavior is observed for some of the Baker--Luecke asymmetric L-space knots.

    We also classify the quasi-alternating surgeries on torus knots and show that the set of formal L-space slopes is either empty or infinite This allows us to give examples of asymmetric formal L-spaces.
\end{abstract}

	
	\makeatletter
	\@namedef{subjclassname@2020}{%
		\textup{2020} Mathematics Subject Classification}
	\makeatother%For 2020
	
	 % Mathematical subject classification
	
	% 57K10 Knot theory
	% 57K14 Knot polynomials
	% 57K16 Finite-type and quantum invariants, topological quantum field theories (TQFT)
	% 57K32 Hyperbolic 3-manifolds
	% 57M12 Low-dimensional topology of special (e.g., branched) coverings
	% 57R58 Floer homology
	% 57R65 Surgery and handlebodies
	
	\maketitle
	
	\tableofcontents
	
\section{Introduction}
\addtocontents{toc}{\protect\setcounter{tocdepth}{1}}%For hiding subsections of introduction in table of content
A knot $K$ in $S^3$ is called an \textit{L-space knot} if $K$ admits a Dehn surgery to a Heegaard Floer L-space~\cite{OS05}.\footnote{A $3$-manifold $M$ is an L-space if its Heegaard Floer homology groups fulfill $\rk(\widehat{\HF}(M))=\chi(\widehat{\HF}(M))$, cf.~\cite{kotelskiy2022thin}.}
		
All known examples of L-space knots in $S^3$ admit a non-trivial Dehn surgery to an L-space that is the double branched cover of a link in $S^3$. Indeed, the first known examples of asymmetric\footnote{Recall that a $3$-manifold is called asymmetric if its mapping class group is trivial. A knot is called asymmetric, if its exterior is asymmetric.} L-space knots in $S^3$ were shown to be L-space knots because they admitted a Dehn surgery to the double branched cover of an alternating link \cite{BL17}. It is thus natural that one should wish to study when surgeries on L-space knots are a double branched cover of a link in $S^3$ and, given such a surgery, what we can say about the branching sets.
	
Let $K$ be a knot in $S^3$, $r\in \Q$ be any non-trivial slope, and $K(r)$ be the result of Dehn surgery on $K$ with slope $r$. 
The slope $r$ is an \textit{L-space slope} or a \textit{dbc slope} if $K(r)$ is an L-space or the double branched cover of a link in $S^3$, respectively.  
Similarly, the slope $r$ is an \textit{alternating slope}, a \textit{quasi-alternating slope}, or a \textit{thin slope} if $K(r)$ is the double branched cover of a non-trivial link in $S^3$ that is alternating, quasi-alternating, or Khovanov thin, respectively. See Section~\ref{sec:definitions} for the definitions of these terms.
Using this language, the proper containments of links
	\begin{quote}
		\{non-split alternating\} $\subsetneq$ \{quasi-alternating\} $\subsetneq$ \{Khovanov thin\}
	\end{quote}
yield the following implications for a slope of a non-trivial knot:
	\begin{quote}
	 alternating\footnote{The double branched cover of a split link is reducible with an $S^1 \times S^2$ summand. However,
Property R~\cite{Ga87} implies that only the unknot has such a surgery among knots in $S^3$.}
  $\implies$ quasi-alternating $\implies$ thin $\implies$ dbc $\cap$ L-space\footnote{It follows from~\cite{OS05} that the double branched cover of a thin knot is an L-space, see for example Theorem 4.5 in~\cite{Marengon_thesis}.}
	\end{quote}
    
In particular, any knot with an alternating, quasi-alternating, or thin slope is an L-space knot.

\medskip
Here we concentrate on quasi-alternating surgeries, since alternating surgeries are studied extensively in \cite{BKM_alt}. Quasi-alternating surgeries are significantly harder to understand and classify than alternating surgeries, which are amenable to study via the use of changemaker lattices \cite{Gr13,GreeneLRP,Gr14,Greene2015genusbounds,Gi15,Mc15,Mc16,McCoyAltUnknotting,Mc17}.

In particular,  we develop methods in \S \ref{sec:strategies} to search for quasi-alternating slopes on hyperbolic knots and use them in \S \ref{sec:asym_census} and \S \ref{sec:sym_census} to determine which knots in the SnapPy census admit a quasi-alternating surgery.  Notably, this enables the completion of the classification of L-space knots in the census, see \S \ref{sec:classification}.  We overview these results in \S \ref{subsec:QAsurgeries-overview}.

Our methods work particularly well for asymmetric hyperbolic knots.   We classify the quasi-alternating slopes of the asymmetric census knots in \S \ref{sec:asym_clasification_dbc} and explore the quasi-alternating slopes of the asymmetric Baker--Luecke knots in \S\ref{sec:Baker_Luecke}.  These results, overviewed in \S\S \ref{subsec:dbcslopeasymmcensus-overview}--\ref{subsec:BLknot-overview}, suggest that potentially any asymmetric hyperbolic L-space knot has exactly two quasi-alternating surgeries and their slopes are consecutive integers, see Question~\ref{ques:asym}.

\medskip
 
We further examine surgeries yielding formal L-spaces in \S \ref{sec:FLSsurgeries}.
Introduced in \cite{Greene2016Strong}, formal L-spaces are the natural 3-manifold analogues of quasi-alternating links, see \S\ref{sec:formal}. Indeed, the double branched cover of a quasi-alternating link is always a formal L-space.

Aligned with our language above, a slope $r$ of a knot $K$ is a \textit{formal L-space slope} if $K(r)$ is a formal L-space.  This yields the following implications for a slope of a non-trivial knot:
	\begin{quotation}
	 alternating
  \hspace{-.1cm}$\implies$\hspace{-.1cm} quasi-alternating \hspace{-.1cm}$\implies$\hspace{-.1cm} formal L-space \hspace{-.1cm}$\implies$\hspace{-.1cm} L-space
	\end{quotation}
As overviewed in \S \ref{subsec:formalLspace-overview}, we will see that in general the set of formal L-space slopes for a knot is distinct from the sets of quasi-alternating slopes, thin slopes, and L-space slopes.

\subsection{Existence of quasi-alternating surgeries on the census knots}\label{subsec:QAsurgeries-overview}
This article grew from an exploration of the alternating and quasi-alternating slopes of Dunfield's \textit{census knots}, the $1267$ hyperbolic knots in $S^3$ with complements in the SnapPy census of hyperbolic cusped manifolds that can be triangulated with at most nine ideal tetrahedra \cite{Du18}. Almost half of the census knots are L-space knots~\cite{Du19}, and  they provide a good testing ground for properties of L-space knots cf.~\cite{ABG+19,BKM_alt,BKM_thin,CensusKnotInvariants}.

More precisely: Among the $1267$ census knots, Dunfield found that $630$ are L-space knots and $635$ are not L-space knots, leaving $2$ with their L-space status undetermined. These are the knots $o9\_30150$ and $o9\_31440$. 
By demonstrating that they both admit quasi-alternating surgeries, we confirm that these two knots are L-space knots and thereby complete Dunfield's classification of the L-space knots among the census knots.

\begin{thm:2uncertain}
The knots $o9\_30150$ and $o9\_31440$ both have a quasi-alternating surgery. Hence they are L-space knots.  
\end{thm:2uncertain}

Of these census L-space knots, $9$ are asymmetric and the remaining $623$ have a strong involution as their only symmetry.  
 
Our main result shows that almost every census L-space knot admits at least one quasi-alternating slope.
	
\begin{thm:QAsurgeries}[Abbreviated]
	Every census L-space knot admits a quasi-alternating surgery, except for the knots $t09847$ and $o9\_30634$ which do not admit any quasi-alter\-nating surgeries.
\end{thm:QAsurgeries}

The full statement of the theorem points out that many of the census L-space knots have surgeries to Seifert fibered spaces that are double branched covers of quasi-alter\-nating Montesinos links.
Table~\ref{tab:QATable} displays a quasi-alternating surgery for each census L-space knot that does not admit an alternating slope or a quasi-alternating Seifert fibered slope. For the census L-space knots with an alternating surgery, we refer to~\cite{BKM_alt}. The Seifert fibered surgeries can be found at~\cite{Du18} or~\cite{BKM}. 

\subsection{Classification of dbc slopes of the asymmetric L-space census knots}\label{subsec:dbcslopeasymmcensus-overview}
For a hyperbolic knot $K$, a slope $r$ is an \textit{exceptional slope} if $K(r)$ is non-hyperbolic and a \textit{symmetry-exceptional slope} if $K(r)$ is hyperbolic and the symmetry group of $K(r)$ does not inject into the symmetry group of the complement of $K$.  
Work of Futer--Purcell--Schleimer \cite{FPS19} enables the computation of a finite list of slopes for a knot $K$ that contains any exceptional or symmetry-exceptional slope, see Corollary~\ref{cor:computableexceptionalsymmetryslopes}.  Since a dbc slope for an asymmetric knot must be an exceptional or symmetry-exceptional slope, this significantly whittles down the search for quasi-alternating slopes of the asymmetric census L-space knots.  We determine the symmetry-exceptional slopes for the asymmetric census L-space knots in Theorem~\ref{thm:exceptionalsym} and present them in Table~\ref{tab:symmetric}. Ultimately this yields:
	
\begin{thm:asym_QA}
	Each asymmetric census L-space knot
	admits exactly two dbc slopes, 
	these two dbc slopes are consecutive integers, 
	and each of these dbc slopes is quasi-alternating but not alternating.
	These slopes are presented in Tables~\ref{tab:symInfo1} and~\ref{tab:symInfo2}.
\end{thm:asym_QA}

In particular, this answers Question 12 from~\cite{ABG+19}.

For the remaining $623$ census L-space knots that are strongly invertible, the classification of quasi-alternating slopes is not so forthcoming as the set of their slopes that are dbc and L-space is infinite.  Indeed, every slope on a strongly invertible knot is a dbc slope, and for an L-space knot $K$, mirrored if needed to make it a {positive} L-space knot, the set of L-space slopes is given by $[2g(K)-1,\infty)$ \cite{genus_L_spaces}.  

\subsection{The Baker--Luecke asymmetric L-space knots}\label{subsec:BLknot-overview}
It seems surprising that every asymmetric census knot has exactly two dbc slopes, and these slopes are consecutive integers and both quasi-alternating. One might wonder if this is a general phenomenon, see Question~\ref{ques:asym}. Other asymmetric L-space knots were con\-structed in~\cite{BL17}. In~\cite{BKM_alt} we have shown that the simplest $14$ of these knots have exactly two alternating slopes and they are consecutive integers. Here we generalize these results as follows.

\begin{thm:BLknots}[Abbreviated]
		For $6$ of the $14$ simplest asymmetric hyperbolic L-space knots from~\cite{BL17}, the set of dbc slopes agrees with the set of alternating slopes and consists of two consecutive integers. For the other $8$ knots, the set of dbc slopes contains two consecutive integers which are both alternating slopes and we identify an explicit finite set of further slopes that might be dbc slopes.
	\end{thm:BLknots}
 
\subsection{The classification of quasi-alternating slopes on torus knots}	\label{subsec:torusknotQA-overview}
Torus knots punctuate the discrepancy between quasi-alternating slopes and thin slopes.
Using results of Issa~\cite{Is18}, we classify the set of quasi-alternating slopes of torus knots.
	\begin{thm:torus_knotsQA}
		A slope $p/q\in\Q$ is quasi-alternating for the torus knot $T_{a,b}$ with $a>b>0$ if and only if
  \begin{equation*}
			\frac{p}{q}>ab-\max\left\{\frac{a}{m},\frac{b}{n}\right\},
		\end{equation*}
  where integers $m,n$ satisfy $bm+an=ab+1$ and $1\leq m<a$ and $1\leq n <b$. 
\end{thm:torus_knotsQA}

\begin{ex}  \label{ex:torus}
	Using Theorem~\ref{thm:torus_knotsQA} with  $a=5$, $b=3$, and choosing $m=2$, %$c=-2$, 
 one finds that a slope $p/q$ for $T_{5,3}$ is quasi-alternating if and only if $p/q>12+1/2$.  
	However, $12+1/2$ is a thin slope:
	$T_{5,3}(12+1/2)$ is the double branched cover of the knot $K11n50$ which is thin but not quasi-alternating~\cite{Greene_thin_QA}. The classification of thin slopes remains open, see Example \ref{ex:torus_thin} and Question~\ref{ques:thin_torus}.
\end{ex}

\subsection{Formal L-space surgeries}\label{subsec:formalLspace-overview} Since the double branched cover of a quasi-alternating link is an example of a formal L-space (see Section~\ref{sec:formal} for the definition), it is natural to compare the formal L-spaces surgeries to the quasi-alternating surgeries. It turns out that the set of formal L-space slopes is always either empty or infinite. This is in contrast to both the quasi-alternating and the alternating slopes which may be non-empty and finite.

\begin{thm:plusone}
    If $K(r)$ is a formal L-space for some $r>0$, then $K(r +n)$ is a formal L-space for every $n\in \Z_{\geq 0}$.
\end{thm:plusone}

This allows us to exhibit the first known examples of formal L-spaces that are not the double branched cover of a quasi-alternating link. 

\begin{cor:asymmetric_formal_L_spaces}
    There exist infinitely many asymmetric formal L-spaces. In particular, there are infinitely many formal L-spaces that are not the double-branched cover of any link in $S^3$.
\end{cor:asymmetric_formal_L_spaces}

\begin{rem}
    In the proof of Corollary~\ref{cor:asymmetric_formal_L_spaces} we will see that there exist infinitely many formal asymmetric L-spaces that arise by surgery on a knot. In particular, there exist formal L-space slopes that are not thin slopes. On the other hand, the slope $12+1/2$ is thin for $T_{5,3}$ as discussed in Example~\ref{ex:torus}, but in~\cite{Greene_thin_QA} it is shown that this slope is not a formal L-space slope. Thus thin slopes and formal L-space slopes are in general different.
\end{rem}

In the presence of integral formal L-space slopes, we have a stronger result.

\begin{thm:formal_integer}
    If $K(n)$ is a formal L-space for some $n\in \Z_{>0}$, then $K(r)$ is a formal L-space for every $r\geq n$.
\end{thm:formal_integer}

In general, calculating the full set of formal L-space slopes seems challenging. However, for torus knots, it turns out to coincide with the set of quasi-alternating surgeries.

\begin{cor:formal_torus}
    For positive torus knots, the set of formal L-space surgeries agrees with the set of quasi-alternating surgeries given in Theorem~\ref{thm:torus_knotsQA}.
\end{cor:formal_torus}

Recall from \cite{OS05b} that if $K$ is an L-space knot, then $K(r)$ is an L-space if and only if $r\geq 2g(K)-1$. Thus the following result implies that every non-trivial L-space knot admits L-space slopes that are not formal L-space slopes. In particular, there is no non-trivial L-space knot for which every L-space surgery is a quasi-alternating surgery.

\begin{prop:formal_bound}
    Let $K$ be a non-trivial knot and $r>0$ be a formal L-space surgery. Then
    \[ r > 2g(K) + \frac{1}{2}\left(\sqrt{1+8g(K)}-1\right).\]
\end{prop:formal_bound}

\addtocontents{toc}{\protect\setcounter{tocdepth}{2}}

\bigskip

\begin{center}$\bullet$\end{center}

\bigskip

\subsection*{Conventions for slopes}
In general, a slope is an unoriented essential simple closed curve in a torus.  With respect to an oriented basis $\alpha,\beta$ for the first homology of the torus, a slope corresponds to a rational number $p/q \in \Q \cup \infty$ or the pair $(p,q)$ for coprime integers $p,q$ if it is homologous to $p\alpha+q\beta$ for some choice of orientation on the curve.

Surgery slopes of a knot $K$ in $S^3$ are slopes in $\bdry \nbhd(K)$,  the boundary torus of a tubular neighborhood of $K$.  Usually, these are measured with respect to the \textit{$S^3$-basis} $\mu,\lambda$ where $\mu$ denotes the meridian of $K$ and $\lambda$ the Seifert longitude of $K$.  (An orientation chosen on $K$ orients $\mu$ and $\lambda$ so that 
$\mu$ links $K$ positively once and the pair $(\mu,\lambda)$ yields the positive orientation of $\bdry \nbhd(K)$.)

Slopes in a cusp of a hyperbolic manifold are often measured with respect to the \textit{geometric basis} of the cusp given by the two shortest slopes.  

SnapPy uses the geometric basis for cusps of census manifolds (and other hyperbolic manifolds without an alternative basis being implicit).  For census knots,  tabulated links,  and manifolds obtained from link diagrams, SnapPy uses the $S^3$-basis. We will follow this convention and use different slope indices according to how the knot/link is described.
For example, the complement of the five-crossing twist knot $K5a1$ (in the Hoste--Thistlethwaite notation) is isometric to the census manifold $m015$ while the slope $p/q$ for $K5a1$ is the slope $p/(q-2p)$ for $m015$.

To give surgery descriptions of quotients of orientation
preserving involutions of manifolds, we use the Hoste--Thistlethwaite--Weeks tabulation of knots and links \cite{HTW} (in short: the HTW table) as recorded in SnapPy which gives an ordering to the link components. As in SnapPy, surgery information is presented in a list-without-commas of ordered pairs $(p_i, q_i)$ indicating $p_i/q_i$ surgery on the $i$th component (with respect to the $S^3$-basis).  Here the ordered pair $(0,0)$ is taken to mean that no surgery is performed on the component and it forms part of the branching set.  For example, $L11n350(0,0)(-4,1)(0,0)$ describes a link of two components arising from the first and third component of $L11n350$ in the manifold obtained by $-4$ surgery on the second component.
When no surgery information is given, the link is just the link in $S^3$.
	
\subsection*{Code and data}
The code and additional data accompanying this paper can be accessed at~\cite{BKM}.	
	
\subsection*{Acknowledgments}
We thank Dave Futer, Liam Watson, and Claudius Zibrowius for useful discussions, explanations, and remarks. For the computations performed in this paper, we use and combine KLO (Knot like objects)~\cite{Sw}, KnotJob~\cite{KnotJob}, the Mathematica knot theory package~\cite{knotatlas}, Regina~\cite{BBP+}, sage~\cite{sagemath}, SnapPy \cite{CDGW}, and code developed by Dunfield in~\cite{Du18,Du19}. We thank the creators for making these programs publicly available. We are also grateful to the three anonymous referees for careful reading, useful suggestions, and comments.

KLB was partially supported by the Simons Foundation grant \#523883 and gift \#962034.

MK was supported by the SFB/TRR 191 \textit{Symplectic Structures in Geometry, Algebra and Dynamics}, funded by the DFG, German Research Foundation (Projektnummer 281071066 - TRR 191), by the DFG (Project: 561898308); by a Ram\'on y Cajal grant (RYC2023-043251-I) and by the project PID2024-157173NB-I00 funded by MCIN/AEI/10.13039/501100011033, ESF+ and FEDER, EU; and by a VII Plan Propio de Investigación y Transferencia (SOL2025-36103) of the University of Sevilla.

DM is supported by NSERC and a Canada Research Chair.
\section{Prerequisites}	
	
\subsection{Terminology}
\label{sec:definitions}
In this section, we briefly explain our conventions and terminology.
	
\subsubsection{Tangle exteriors}
Here, a \textit{tangle exterior} $T$ of a knot $K$ in $S^3$ is a  two-strand tangle in a $3$-ball whose double branched cover is the knot exterior $S^3\cut \nbhd(K)$. Then, via the Montesinos Trick \cite{Montesinos_trick}, the Dehn surgery $K(r)$ on $K$ by any slope $r$ is diffeomorphic to the double branched cover $\dbc(T(r))$ of the tangle filling $T(r)$ of $T$ by a rational tangle of slope $r$. (The parametrization of slopes for the rational tangle filling of the tangle exterior is inherited from the parametrization of slopes for the Dehn filling of the knot exterior, see for example~\cite{Watson_KH_sym2017} for more details.)

\subsection{Thin links}
A link $L$ is called \textit{thin}, if the Khovanov homology $Kh^{h,q}(L)$ (with $\Z_2$-coefficients) of $L$ is supported in two diagonal gradings $\delta = q - 2h$. 
While the Khovanov homology of a link depends on the orientation of its components, the thinness status of a link is independent of the orientations of the link components. This follows, for example, from Proposition 28 in~\cite{Khovanov_homology}. A link that is not thin is called \textit{thick}.
	
\subsubsection{Quasi-alternating links}
The set of \textit{quasi-alternating links} $QA$ is the smallest set of links such that
\begin{enumerate}
    \item the unknot is in $QA$, and
    \item if a link $L$ has a crossing $c$ in a diagram $D$ such that both smoothings $D_0$ and $D_1$ at $c$ represent links $L_0$ and $L_1$ in $QA$ with $\det(L_0)+\det(L_1)=\det(L)$, then $L$ is also in $QA$.
\end{enumerate}
From the definition, it is straightforward to see that any non-split alternating link is quasi-alternating. Furthermore, any non-split quasi-alternating link is Khovanov thin (as well as knot Floer homology thin), implying that its double branched cover is an L-space~\cite{OS05,QA_thin}. On the other hand, there exist quasi-alternating knots that are not alternating~\cite{KnotInfo}, thin knots that are not quasi-alternating~\cite{Greene_thin_QA}, cf.~\cite{GW13}, and thick knots whose double branched cover is an L-space, see for example~\cite{Is18}.\footnote{As pointed out by a referee, a simple example of a thick know whose double branched cover is an L-space is given by the torus knot $T_{3,5}$.}

\subsubsection{Formal L-spaces}\label{sec:formal}

A \textit{triad} is a triple of closed, oriented 3-manifolds $(Y_0, Y_1, Y_2)$ such that there exists
 a 3-manifold $M$ with torus boundary and a collection of slopes $\gamma_0, \gamma_1, \gamma_2$ in $\partial M$ of
pairwise distance one such that $Y_i$
is the result of Dehn filling $M$ along $\gamma_i$.
\begin{defi}[\cite{Greene2016Strong}]
    The set of \textit{formal L-spaces} $\F$ is the smallest set of rational homology 3-spheres such that 
\begin{enumerate}
    \item $S^3\in \F$ and
    \item whenever $(Y, Y_0, Y_1)$ is a triad with $Y_0\in \F$, $Y_1\in \F$ and
    \[
    |H_1(Y;\Z)|=|H_1(Y_0;\Z)|+|H_1(Y_1;\Z)|
    \]
    we have $Y\in \F$.
\end{enumerate}
\end{defi}

\begin{rem}\label{rem:formal_Lspace_QA}
     Formal L-spaces can be seen as the $3$-manifold analog of quasi-alternating links. In both cases, one obtains the corresponding class by starting from the simplest object (the unknot in the link setting and $S^3$ in the 3-manifold
setting) and then closing under a recursive operation governed by additivity of 
determinant or homology in a triad. For quasi-alternating
links, the skein triple $(L,L_0,L_1)$ at a crossing plays the role of a triad with the determinant relation $\det(L)=\det(L_0)+\det(L_1)$ mirroring
the homological relation
\[
|H_1(Y;\mathbb{Z})| = |H_1(Y_0;\mathbb{Z})| + |H_1(Y_1;\mathbb{Z})|
\]
in a Dehn filling triad. Since the determinant of a link gives the order of the first homology of its double branched cover it follows that the double branched cover of a quasi-
alternating link is a formal L-space.
 Thus, the definition of formal L-spaces
can be viewed as a natural 3-manifold analogue of quasi-alternating
links.
\end{rem}
	
\subsection{Symmetry groups of hyperbolic knots}\label{sec:symmetry_groups}
Let $M$ be a $3$-manifold. The \textit{symmetry group} of $M$ is defined to be the mapping class group of $M$, i.e.\ the group of orientation-preserving diffeomorphisms of $M$ that fix the boundary pointwise up to isotopy that fixes the boundary pointwise. If $M$ is hyperbolic then the symmetry group of $M$ is isomorphic to its isometry group~\cite{Gabai_DiffMCG,Gabai_DiffMCG2}. The symmetry group of a hyperbolic knot (i.e.\ the symmetry group of its complement) is known to be a subgroup of a dihedral group~\cite{HTW_table} and thus it is $\Z_n$, $D_n$, or $\Z_2^2$.
 
An orientation-preserving involution on $S^3$ is either the antipodal map, which has no fixed set, or a $\pi$-rotation about the fixed set of a great circle \cite{Waldhausen_involutions}. 
A \textit{free involution} on a knot $K$ is an involution of $(S^3,K)$ that acts without fixed points on $K$.  These come in two types depending on whether or not the action has fixed points in $S^3$.
A \textit{strong inversion} on a knot $K$ is an involution of $(S^3, K)$ with a circle fixed set that meets $K$ in two points.  
In both situations, $K$ has a tubular neighborhood that is equivariant.  
Observe that the quotient of a strong involution of $K$, restricted to the exterior of an equivariant tubular neighborhood of $K$, is a tangle exterior for $K$~\cite{Montesinos_trick}. 
	
\begin{lem}\label{lem:involutions_of_knot1}
If the complement of a hyperbolic knot $K$ admits an orientation-preserving involution $f$, then $f$ extends to an involution of $(S^3,K)$.
\end{lem}
	
\begin{proof}
Let $f$ be an orientation-preserving involution on the exterior $M_K=S^3\cut\nbhd(K)$ of $K$.  Then $f$ must preserve the slope of the longitude of $K$ in $\bdry \nbhd(K)$: $f(\lambda) = \pm \lambda$.  Furthermore, due to the solution of the knot complement problem~\cite{GL89}, $f$ must also preserve the slope of the meridian of $M$: $f(\mu)=\pm\mu$. Since $f$ is orientation-preserving on $M_K$, the restriction of $f$ to $\bdry\nbhd(K)$ is either the identity $I$ or its negation $-I$.  In both cases, this map $f$ extends to an involution of $S^3$ that preserves $K$. The former gives a free involution while the latter gives a strong inversion.
\end{proof}

\subsection{Double branched covers}\label{sec:DBC}
Let $M$ be a closed $3$-manifold with a smooth orientation-preserving involution $\varphi\colon M\rightarrow M$. 
We can take the quotient under $\varphi$, i.e.\ we identify points $p$ and $\varphi(p)$. Then the quotient map $M\rightarrow M/\varphi$ will be a (branched or unbranched) $2$-fold covering map with (possibly empty) branching set $B_\varphi$. For this, we also write $M=\operatorname{DBC}(B_\varphi)$. It will be an unbranched covering if and only if the involution acts freely. The following lemma says that the quotient manifold and its branching set only depend on the isotopy class of $\varphi$ in the mapping class group of $M$.

\begin{lem}\label{lem:dbc_lemma}
Suppose $M$ is a closed hyperbolic $3$-manifold with two smooth orienta\-tion-preserving involutions $\varphi_1,\varphi_2\colon M\rightarrow M$ that are isotopic as diffeomorphisms. Then $(M/\varphi_1,B_{\varphi_1})$ is diffeomorphic to $(M/\varphi_2,B_{\varphi_2})$.
\end{lem}

\begin{proof}
  We see $M/\varphi_i$ as an orbifold $\mathcal{O}_i$. Then the quotient map $M\rightarrow \mathcal O_i$ is a $2$-fold orbifold covering map. Since $M$ is hyperbolic, the orbifold universal covering of $\mathcal O_i$ is hyperbolic, and thus $\mathcal O_i$ is irreducible and atoroidal. Then the orbifold geometrization theorem~\cite[Corollary~1.2]{BLP_annals} 
  implies that $\mathcal O_i$ is hyperbolic and the deck transformation $\varphi_i$ is an isometry for some hyperbolic metric $g_i$ on $M$. Now Mostow rigidity implies that there exists an isometry $f\colon (M,g_1)\rightarrow (M,g_2)$ such that $f^{-1}\varphi_2 f=\varphi_1$. But then $\mathcal O_1\ni [p]\mapsto [\varphi_2f(p)]\in\mathcal O_2$ is a diffeomorphism respecting the branching sets.
\end{proof}
	
In this work, the only non-trivial symmetry groups of closed hyperbolic $3$-ma\-ni\-folds that we need to consider are the symmetry groups $\Z_2$ and $\Z_2^2$.  
In the former case, there is a unique involution.
In the latter case, the symmetry group is generated by two different involutions $f$ and $g$. The involution $f\circ g$ will always be different from $f$ and $g$ since it represents a different element in the symmetry group. (However, $g\circ f$ is isotopic to $f\circ g$.) 
If we present our closed hyperbolic $3$-manifold via a surgery diagram, we can often explicitly determine each involution via symmetries of the surgery link and then explicitly determine the corresponding quotient and branch set there.  For a more detailed discussion see~\cite{Au14}. 
	
\subsection{Length of geodesics after surgery and symmetry-exceptional slopes}\label{sec:bound}
In this section, building upon \cite{FPS19}, we will give an explicit bound $C$ such that any surgery on a hyperbolic knot along a slope whose length is larger than $C$ yields again a hyperbolic manifold with the same symmetry group as the knot. We first set up the notation.
	
Let $M$ be a complete hyperbolic $3$-manifold with finite volume and one cusp. We denote by $\operatorname{systole}(M)$ the length of the shortest geodesic of $M$. 
There is a \textit{(normalized) length} $L(r)$ associated to each slope $r$ on $\bdry M$. This is obtained by taking a horospherical neighbourhood $N$ of the cusp with the natural inherited Euclidean metric on $N$ and setting
\begin{equation*}
    L(r)=\frac{\operatorname{length}(r)}{\sqrt{\operatorname{area}(\partial N)}}.
\end{equation*}
where $\operatorname{length}(r)$ is the minimal length of all curves representing $r$ on $\partial N$. Since changing the horospherical neighbourhood $N$ has the effect of rescaling the Euclidean metric on $\partial N$, this $L(r)$ is independent of the choice of $N$.

We denote by $M(r)$ the Dehn filling of $M$ along the slope $r$.  If the slope $r$ is in the hyperbolic Dehn surgery space (so that $M(r)$ has a hyperbolic structure obtained from deforming the one on $M$) then we write $\gamma$ for the geodesic core of the newly glued-in solid torus in $M(r)$. From Thurston's hyperbolic Dehn surgery theorem~\cite{Th79}, it follows that $\gamma$ will be the shortest geodesic in $M(r)$ provided that $L(r)$ is sufficiently large. 
	
The following theorem, which is a slight reformulation of Theorem~7.28 in~\cite{FPS19}, hands us an explicit such bound on the length of the slope.
	
\begin{thm}[Futer--Purcell--Schleimer~\cite{FPS19}]\label{thm:FPS}
If $C(M)$ is given by
\begin{equation*}
C(M)= \max\left\{10.1\,,\,\sqrt{\frac{2\pi}{\operatorname{systole}(M)}+58}\right\},
\end{equation*}
then for any surgery slope $r$ with
\begin{equation*}
L(r) \geq C(M)
\end{equation*}
the geodesic core $\gamma$ of the filling solid torus is the shortest geodesic in $M(r)$. \qed
\end{thm}
	
Note that $L(r)\geq C(M)>6$, and thus the manifold $M(r)$ is necessarily hyperbolic by~\cite{Ag00, La00}, cf.~\cite{BH96}. Note that the $6$-theorem is in terms of non-normalized length, so we are implicitly using the fact that there always exists a cusp of area at least one, see Theorem~1.2 in~\cite{gabai2021hyperbolic}.
	
We immediately get the following corollary about symmetry-exceptional surgery slopes for knots.
	
\begin{cor}\label{cor:slope}
Let $M$ be the complement of a hyperbolic knot. If $r$ is a surgery slope such that $L(r)\geq C(M)$ where $C(M)$ is the constant from Theorem~\ref{thm:FPS}, then every symmetry of $K(r)$ restricts to a symmetry of the complement of $K$. In particular, if $r$ is a symmetry-exceptional slope, then $L(r) < C(M)$.
\end{cor}
	
\begin{proof}
When $L(r)\geq C(M)$, Theorem~\ref{thm:FPS} says that (in addition to $M(r)$ being the hyperbolic manifold obtained by deforming the hyperbolic structure on $M$)  the core $\gamma$ of the newly glued-in solid torus is the shortest geodesic in $M(r)$. It follows that any element in the symmetry group of $M(r)$ has to send $\gamma$ to itself and thus restricts to a symmetry of the knot complement $M$. 
\end{proof}
	
Since there are only finitely many slopes in a cusp with length smaller than a given constant, we immediately have:
\begin{cor}\label{cor:computableexceptionalsymmetryslopes}
Let $M$ be the complement of a hyperbolic knot.  There is a computable finite list of slopes along which a Dehn surgery may produce a manifold with more symmetries than the knot. \qed
\end{cor}
	
In particular, when classifying the quasi-alternating slopes for any asymmetric knot,  we only need to check the finite list of slopes given by Corollary~\ref{cor:computableexceptionalsymmetryslopes}. 
	
\section{Algorithms and Strategies} \label{sec:strategies}
	
In this section, we discuss certain strategies and algorithms that we use later to prove our main theorems. We will apply these strategies to certain hyperbolic L-space knots whose symmetry groups are either trivial or generated by a single strong inversion. It is conjectured that this covers all symmetries of hyperbolic L-space knots \cite[Conjecture 1.7]{L_space_periodic_conj}. Thus we restrict the strategies to such knots whenever that simplifies the discussion. 
	
 \begin{strat}[Exceptional slopes]\label{strat:exceptional_slopes}
Classify the exceptional slopes and their fillings of a hyperbolic knot $K$.
\begin{enumerate}
\item Using the cusp shape of $K$, make a list of short slopes of length at most $6$. Every exceptional slope is contained in this finite list by the 6-Theorem \cite{Ag00, La00}.
\item For each short slope $r$: 
\begin{enumerate}
\item Check with SnapPy if a hyperbolic structure is found on $K(r)$. If so, the slope $r$ is not exceptional.
\item If no hyperbolic structure is found, check with Regina and Dunfield's recognition code~\cite{Du18} for classification of non-hyperbolic type. If this recognizes $K(r)$ as non-hyperbolic, $r$ is an exceptional slope.  
\item If neither succeeds, leave the slope $r$ with undetermined exceptional status. 
\end{enumerate}
\end{enumerate}
\end{strat}
	
\begin{rem}
Dunfield employs such a strategy to classify exceptional slopes of census knots \cite{Du18}. Since it is algorithmically decidable if a manifold is hyperbolic~\cite{decide_hyperbolic}, this strategy yields an algorithm that is always guaranteed to terminate. 
\end{rem}
	
\begin{strat}[Symmetry-exceptional slopes]
\label{strat:symmetryexceptional_slopes}
Classify the symmetry-exceptional slopes of a hyperbolic knot $K$.
\begin{enumerate}
\item Compute the symmetry group of $K$.
\item Using the cusp shape and systole computations of $K$, make a list of short slopes of length less than the FPS bound, defined in Theorem~\ref{thm:FPS}. By Corollary~\ref{cor:slope} this finite list contains all symmetry-exceptional slopes.
\item Use Strategy~\ref{strat:exceptional_slopes} to identify the exceptional slopes. 
\item For each remaining short slope $r$:
\begin{enumerate}
\item Check if SnapPy finds a hyperbolic structure on $K(r)$ and gives a verified computation of the symmetry group.
\item If so, compare with the symmetry group of $K$ to determine if $K$ is symmetry-exceptional or not.
\item If not, leave the slope $r$ with undetermined symmetry-exceptional status. 
\end{enumerate}
\end{enumerate}
\end{strat}

\begin{rem}\label{rem:sym}
This strategy is not theoretically always guaranteed to terminate. The problem is that there is no algorithm known that is guaranteed to compute the symmetry group of a hyperbolic manifold. However, in practice, the methods implemented in SnapPy usually terminate and yield verified computations of symmetry groups. 
\end{rem}
	
\begin{strat}[Branch set search]
\label{strat:branchlocussearch}
Given a closed oriented hyperbolic manifold $Y$, search for all presentations of $Y$ as the double cover of a manifold $M$ branched over a link $L$ in $M$.	
\begin{enumerate}%[label*=\arabic*.]
    \item Using the symmetry group of $Y$, compute the number $n$ of non-isotopic orientation-preserving involutions of $Y$. By Lemma~\ref{lem:dbc_lemma}, $Y$ can be written in exactly $n$ different ways as double branched cover over a manifold $M$ branched along a link $B$ in $M$.
    \item Start with an empty list  $\mathcal{SD}(Y)$ of surgery descriptions for links $B$ in a $3$-manifold $M$ whose double branched cover is $Y$.
    \item Iterate through links in 3-manifolds by taking links in the HTW table and performing Dehn fillings on the components, where some subset of the components receive the trivial $1/0$ filling, measured with respect to the Seifert longitude in $S^3$. This gives a pair $(M,B)$ where $M$ is the surgered manifold and $B$ is the link in $M$ formed by the components with $1/0$ filling.  
Take all double covers of $M$ branched along $B$.\footnote{Note that there might be more than one or no double branched cover depending on the algebraic topology of the link complement.}
For each cover $C$:
    \begin{enumerate}%[label*=\arabic*.]
    \item Check if $C$ is diffeomorphic to $Y$ and if $(M,B)$ is not contained in $\mathcal{SD}(Y)$. 
    \item If so, add the surgery description for $(M,B)$ to $\mathcal{SD}(Y)$.   
    \end{enumerate}	
    \item Stop when $\mathcal{SD}(Y)$ contains $n$ distinct pairs $(M,B)$.
\end{enumerate}		 
\end{strat}

\begin{rem}
    As explained in Remark~\ref{rem:sym}, Step $(1)$ of this algorithm is theoretically not always guaranteed to terminate (since it rests on symmetry group computations) but works well in practice and yields verified results. On the other hand, the rest of the algorithm is theoretically verified to terminate since the decision problem for $3$-manifolds is solved~\cite{Kuperberg_Decision_prob}. But in practice, some of the branching sets may be too complicated to find in a reasonable time. Nevertheless, in most of our examples, the branching sets were sufficiently simple that Strategy~\ref{strat:branchlocussearch} terminated relatively quickly.

    Note that Step $(3)$ and $(4)$ will in theory always terminate and return all branching sets, since every link $B$ in a $3$-manifold $M$ admits a surgery description along a link in $S^3$. However, in practice we iterate only through links in the HTW table with a small subset of chosen surgery coefficients.
\end{rem}
			
\begin{strat}[Tangle exterior search]
\label{strat:tangle_ext}
%SIMPLE VERSION:
Search for a tangle exterior of a strongly invertible knot $K$.
\begin{enumerate}
	\item Using the branch set search \ref{strat:branchlocussearch}, identify for some integer $n\in\Z$ a triple of links $B_n$, $B_{n+1}$ and $B_{n+1/2}$ in $S^3$ such that 
	\begin{align*}
	   \dbc(B_n)=K(n), \,  \dbc(B_{n+1})=K(n+1), \text{and}\, \dbc(B_{n+1/2})=K(n+1/2).
	\end{align*}
	\item Choose a (reduced) diagram $D$ of $B=B_{n+1/2}$.
	\item For each crossing $c$ in the diagram $D$:
	\begin{enumerate}
		\item Check if the two smoothings of $D$ at $c$ are a pair of links with exteriors diffeomorphic to those of $B_n$ and $B_{n+1}$. 
		\item If so, check if changing the crossing at $c$ unknots $B_{n+1/2}$ using the unknot recognition algorithm~\cite{BBP+}. If this is the case, we expect that the exterior of $c$ is the tangle exterior of $K$, since four of its tangle fillings give the correct branching sets.
		\item Verify that the exterior of the crossing $c$ is a tangle exterior of $K$ as follows:
		\begin{enumerate}
			\item Let $\delta$ be a crossing disk for the crossing $c$, i.e.\ an embedded disk $\delta \subset S^{3}$ transversally intersecting the link $B$ at exactly two points in its interior that projects to an arc in the diagram plane meeting $D$ only at the crossing $c$.
			Form the double branched cover $N$ of $S^3-\bdry \delta$ along $B$. 
            \item The Conway sphere $\bdry N(\delta)$ lifts to a torus $F$ in $N$ that bounds a knot manifold to one side and a pair-of-pants $\times S^1$ to the other side. SnapPy~\cite{CDGW} can search for the splitting torus $F$ in $N$ as a normal surface and perform operations on it.
			\item Split $N$ along $F$ and take the piece $N_1$ with one boundary component.   
			\item Check if $N_1$ is isometric to the exterior of $K$. If so, the exterior of the crossing $c$ in the diagram $D$ is the desired tangle exterior.
		\end{enumerate}
	\end{enumerate}
\end{enumerate}
\end{strat}
			
\begin{rem} \hfill    
\begin{enumerate}
	\item The brute force approach of Strategy~\ref{strat:tangle_ext} is not guaranteed to succeed and can certainly be improved in speed and general applicability. For example we could try different integers $n$ in Step $(1)$ or different diagrams $D$ in Step $(2)$. Nevertheless, Strategy~\ref{strat:tangle_ext} works efficiently and successfully to identify a tangle exterior for each strongly invertible census knot. Moreover, because the branch set search (Strategy~\ref{strat:branchlocussearch})  identifies branch loci with low crossing number, the tangle exteriors that Strategy~\ref{strat:tangle_ext} identify also have low crossing number and their integral rational tangle fillings remain rather simple.  
	\item There is an algorithm to identify a tangle exterior for a strongly invertible knot, though it is impractical.  Fix an enumeration by crossing number of all reduced diagrams of 2-strand tangle exteriors, and then inductively check if their double branched covers are diffeomorphic to the exterior of the knot at hand.   This last step uses the solution to the decision problem for diffeomorphisms of 3-manifolds with torus boundary \cite{Kuperberg_Decision_prob}.
\end{enumerate}
\end{rem}

\begin{strat}[DBC slopes]
\label{strat:dbc_slopes}
For a hyperbolic knot $K$ in $S^3$ with trivial symmetry group or symmetry group comprising a single strong involution determine which slopes are dbc slopes and find all their possible branching sets.
\begin{enumerate}
    \item If $K$ is strongly invertible, then every slope is a dbc slope by the Montesinos trick~\cite{Montesinos_trick}. For getting all branching sets we run the tangle exterior search~\ref{strat:tangle_ext} to find the tangle exterior $T_K$ of $K$. Then a branching set of $K(r)$ is given by the rational tangle filling $T_K(r)$.
    \item Any remaining dbc slope of $K$ must be exceptional or symmetry-exceptional. Compute these slopes with Strategies~\ref{strat:exceptional_slopes} and~\ref{strat:symmetryexceptional_slopes}.
    \begin{enumerate}
        \item If $r$ is exceptional, then use either a Seifert fibred structure on $K(r)$ or the Equivariant Torus Theorem~\cite{equivarianttorustheorem} for ad hoc determination. Further discussion on how this can be done in practice can be found in Section 4.1 in~\cite{BKM_thin}.
        \item If $r$ is symmetry-exceptional, perform the branch set search~\ref{strat:branchlocussearch} to account for every involution in $\mathrm{Sym}(K(r))$. (Eg.\ if the involutions generate a $\Z_2^2$ subgroup, find three distinct branch loci.)
    \end{enumerate}
\end{enumerate}
\end{strat}

\begin{rem}
    Strategy~\ref{strat:dbc_slopes} is not verified to terminate since it rests on symmetry group computations, see Remark~\ref{rem:sym}.
\end{rem}
			
\begin{strat}[Quasi-alternating link status]
\label{strat:checkiflinkisQA}
Strategy to check if a link $L$ is quasi-alternating.
				
Begin with lists $\mathcal{QA}$ of known quasi-alternating and $\mathcal{NQA}$ of known non-quasi-alternating links.  Initially, we use Jablan's classification of prime quasi-alternating knots with at most 12 crossings \cite{Ja14} to construct these two lists.   Each time this strategy is run, these lists grow.
\begin{enumerate}
	\item Check if $L$ belongs to either $\mathcal{QA}$ or $\mathcal{NQA}$.
	\item If not, compute the Khovanov homology of $L$.  If it is not thin, then add $L$ to $\mathcal{NQA}$ and stop.
	\item If $L$ is thin, 
	\begin{enumerate}
		\item Search for a diagram $D$ of $L$ with small crossing number.
		\item For each crossing of $D$:
		\begin{enumerate}
			\item Check if the two smoothings at the crossing have determinants that add to the determinant of $D$.
			\item If so, recursively check if both smoothings belong to $\mathcal{QA}$.
			\item If so, add $L$ to $\mathcal{QA}$ and stop.
		\end{enumerate}
	\end{enumerate}
\end{enumerate}
\end{strat}

\begin{rem}
    Since there exist links that are thin but not quasi-alternating~\cite{Greene_thin_QA}, Strategy~\ref{strat:checkiflinkisQA} will in general not terminate. In fact, it remains unknown if it is algorithmically decidable if a link is quasi-alternating or not, see Question~\ref{ques:QA_decidable}. The same holds true for Strategy~\ref{strat:singleQA_slope} below.
\end{rem}
			
\begin{strat}[Single quasi-alternating slope]
\label{strat:singleQA_slope}
Strategy to find a single quasi-alternating slope for a hyperbolic knot $K$ in $S^3$.  
				
Other knowledge about $K$ may rule out the existence of quasi-alternating slopes.
\begin{itemize}
	\item If $K$ is known to not be an L-space knot, then it has no quasi-alternating slopes.  
	\item If $K$ has no Khovanov thin slopes, then it has no quasi-alternating slope.
\end{itemize}
Otherwise:
\begin{enumerate}
	\item Find all dbc slopes $r$ using Strategy~\ref{strat:dbc_slopes}.  
	\item List those with $|r|\geq 2g(K)-1$. (One may use the span of the Alexander polynomial of $K$ in place of $2g(K)$ since they are equal for L-space knots.)
	\item For each slope $r$ on this list:
	\begin{enumerate}
		\item Use Strategy~\ref{strat:branchlocussearch} to search for a branch set $L$ of $K(r)$.
		\item Use Strategy~\ref{strat:checkiflinkisQA} to attempt to identify $L$ as quasi-alternating.
	\end{enumerate} 
\end{enumerate}
\end{strat}
   
\section{The asymmetric census knots}\label{sec:asym_census}	
Among the SnapPy census L-space knots, there are exactly nine knots that are not strongly invertible. These are the asymmetric knots $\calA$ studied in~\cite{ABG+19}:
	\begin{align*}
	\mathcal A = \{t12533, t12681, o9\_38928, o9\_39162, o9\_40363,\\
	o9\_40487, o9\_40504, o9\_40582, o9\_42675\}.
	\end{align*}
			
In this section, we will classify all symmetry-exceptional, dbc, thin, and quasi-alternating slopes of the knots in $\mathcal A$.

Recall, that for the census manifolds we display slopes in the geometric basis as used by SnapPy. From the homology of the fillings it is straightforward to express the geometric basis in terms of the Seifert basis. This is displayed in Table~\ref{tab:slopes}.

\begin{table}[htbp] 
\caption{For each asymmetric census L-space knot we display the relation between the slopes $r_S$ measured with respect to the Seifert basis $(\mu_S,\lambda_S)$ and the slopes $r_G$ measured with respect to the geometric basis $(\mu_G,\lambda_G)$. For all census L-space knots it turns out that $\mu_S=\mu_G$.}
\label{tab:slopes}
\small{
\begin{tabular}{@{}lclclc@{}}
\toprule
knot	  & slope & knot	  & slope & knot	  & slope \\
\midrule
$t12533$ & $r_S=r_G + 37 $ &
$t12681$ &  $r_S=r_G + 61$ &
$o9\_38928$ &   $r_S=r_G + 50$ \\
$o9\_39162$ &  $r_S=r_G + 65$ &
$o9\_40363$ &   $r_S=r_G -83 $ &
$o9\_40487$ &  $r_S=r_G + 38$ \\
$o9\_40504$ &   $r_S=r_G - 58$ &
$o9\_40582$ &   $r_S=r_G +44 $ &
$o9\_42675$ &   $r_S=r_G + 46$ \\
\bottomrule
\end{tabular}
}
\end{table}

\subsection{Symmetry-exceptional slopes for the asymmetric census knots} \label{sec:symmetryexceptionalclassification}
			
In Theorem~\ref{thm:exceptionalsym} below, we classify all symmetry-exceptional slopes of the asymmetric L-space knots in~$\calA$. This will then be used for the proof of Theorem~\ref{thm:asym_QA} in Section~\ref{sec:asym_clasification_dbc} which classifies the quasi-alternating slopes of the asymmetric census knots.
			
\begin{thm} \label{thm:exceptionalsym}
    All exceptional and symmetry-exceptional slopes of the knots in $\calA$ are recorded in Table~\ref{tab:symmetric}. Only three slopes are exceptional, each yielding a small Seifert fibered space. The remaining symmetry-exceptional slopes yield hyperbolic manifolds with symmetry group $\Z_2$ or $\Z_2^2$.
\end{thm}

\begin{proof}
Dunfield has already determined all non-hyperbolic filling slopes for knots in the SnapPy census~\cite{Du18}.
We apply Strategy~\ref{strat:symmetryexceptional_slopes} to compute the symmetry-exceptional slopes of each knot in $\calA$. 
\end{proof}
   	
\begin{table}[htbp] 
\caption{The non-trivial symmetry-exceptional slopes of the asymmetric L-space knots from $\calA$ are shown in SnapPy's basis. They are categorized according to the symmetry group of the filled hyperbolic manifold or the type of non-hyperbolic filling.}
\label{tab:symmetric}
\small{
\begin{tabular}{@{}llcc@{}}
\toprule
knot	  & {$\Z_2$-symmetric} & {$\Z_2^2$-symmetric} & Seifert fibered\\
\midrule
$t12533$ & $(0,1), (1,1),(-1,1), (-2,1),(-1,2),(-3,1)$ &\phantom{$-$}{-}  &{-} \\
$t12681$ & $(-1,1),(-1,2), (-1, 3), (-2, 3)$ &\phantom{$-$}$(1,1)$&$(0,1)$\\
$o9\_38928$ &  $(0,1),(1,1),(-1,1),(-2,1)$ & \phantom{$-$}$(2,1)$&{-} \\
$o9\_39162$ & $(1,1),(-1,1),(2,1),(1, 2),(1, 4)$ &\phantom{$-$}{-} &$(0,1)$\\
$o9\_40363$ &   $(-1,1),(-1, 2),(-1, 3),(-3, 4)$ & \phantom{$-$}$(1,1)$&$(0,1)$\\
$o9\_40487$ & $(0,1),(1,1),(-1,1),(1, 2),(2,1),(3,1),(-2,1)$ &\phantom{$-$}{-} &{-} \\
$o9\_40504$ &  $(0,1),(1,1),(-1,1),(-2,1)$ &\phantom{$-$}$(2,1)$&{-} \\
$o9\_40582$ &  $(0,1),(1,1),(-1,1), (1,2),(2,1),(4,1)$ &\phantom{$-$}{-} &{-} \\
$o9\_42675$ &  $(0,1),(1,1),(-1,1),(2,1)$  &$(-2,1)$ &- \\
\bottomrule
\end{tabular}
}
\end{table}
			
\subsection{Quasi-alternating slopes of the asymmetric census knots}
\label{sec:asym_clasification_dbc}
			
Using Theorem~\ref{thm:exceptionalsym}, we classify all dbc, thin, and quasi-alternating slopes of the asymmetric census knots.
			
\begin{thm}\label{thm:asym_QA}
	Each knot in $\calA$ admits exactly two dbc slopes, these two dbc slopes are consecutive integers, and each of these dbc slopes is quasi-alternating but not alternating. These slopes are presented in Tables~\ref{tab:symInfo1} and~\ref{tab:symInfo2}. 
\end{thm}

\begin{table}[htbp]
{\small
\caption{Information on the symmetric slopes of the asymmetric census L-space knots. A link marked with $(QA)$ is proven to be quasi-alternating. If a symmetry group is marked with $SFS$ the surgered manifold is proven to be a Seifert fibered space.}
\label{tab:symInfo1}
\begin{tabular}{@{}
>{\raggedright}p{0.1\linewidth}
>{\raggedright}p{0.15\linewidth}	
>{\raggedright}p{0.21\linewidth}
>{\raggedright\arraybackslash}p{0.31\linewidth}@{}}
\toprule
slope  & symmetry &  quotient manifold& branching set \\
\midrule
\multicolumn{4}{l}{\textbf{$t12533$}}\\
\midrule
$(0,1)$&$\Z_2$&$S^3$&$K12n407$ (QA)\\
$(1,1)$&$\Z_2$&$-L(2,1)$&$L13n7360(0,0)(2,1)$\\
$(-1,1)$&$\Z_2$&$S^3$&$L12n789$ (QA)\\
$(-2,1)$&$\Z_2$&$-L(3,1)$&$L12n722(3,1)(0,0)$\\
$(-1,2)$&$\Z_2$&$L(3,1)$&$L11n192(-3,1)(0,0)$\\
$(-3,1)$&$\Z_2$&$-L(4,1)$&$L11n350(0,0)(-4,1)(0,0)$\\
\midrule
\multicolumn{4}{l}{\textbf{$t12681$}}\\		
\midrule
 $(0,1)$& SFS &$S^3$&$K11n89$ (QA)\\
 \multirow{3}{*}{$(1,1)$}& 
\multirow{3}{*}{\begin{tabular}{@{}>{\raggedright}p{0.8\linewidth}>{\raggedleft}p{0.1\linewidth}}$\Z_2^2$ & {\Vast\{} \end{tabular}}  %This seems to get the brace in the right place... the linewidth is the width of the cell.
&$-L(5,1)$ &$L11n419(0,0)(5,1)(0,0)$ \\ 
	&  &$-L(3,1)$&$L12n1907(0,0)(3,1)(0,0)$\\ 
	&  &$-L(2,1)$&$L13n7356(0,0)(0,0)(2,1)$\\
 $(-1,1)$&$\Z_2$&$S^3$&$L11n172$ (QA)\\
$(-1,2)$&$\Z_2$&$L(3,1)$&$L10a81(-3,1)(0,0)$\\
$(-1,3)$&$\Z_2$&$L(4,1)$&$L11a282(-4,1)(0,0)$\\
$(-2,3)$&$\Z_2$&$L(5,3)$&$L13n4363(-5,3)(0,0)$\\
\midrule
\multicolumn{4}{l}{\textbf{$o9\_38928$}}\\
\midrule
$(0,1)$&$\Z_2$&$S^3$&$L11n178$ (QA)\\
$(1,1)$&$\Z_2$&$S^3$&$K11n172$ (QA)\\
$(-1,1)$&$\Z_2$&$L(3,1)$&$L12n703(-3,1)(0,0)$\\
$(-2,1)$&$\Z_2$&$L(2,1)$&$L12n1949(0,0)(2,1)(0,0)$\\
\multirow{3}{*}{$(2,1)$}&
\multirow{3}{*}{\begin{tabular}{@{}>{\raggedright}p{0.8\linewidth}>{\raggedleft}p{0.1\linewidth}}$\Z_2^2$ & {\Vast\{} \end{tabular}}
&	$L(3,1)$&$L13n7433(0,0)(-3,1)(0,0)$  \\ 
&&	$L(2,1)$&$L10a105(0,0)(-2,1)$\\ 
&&	$L(8,1)$&$L10n54(-8,1)(0,0)$\\
\midrule
\multicolumn{4}{l}{\textbf{$o9\_39162$}}\\
\midrule
$(0,1)$& SFS &$S^3$&$K12n278$ (QA)\\
$(1,1)$&$\Z_2$&$S^3$&$L12n1050$ (QA)\\
$(-1,1)$&$\Z_2$&$L(3,1)$&$L13n9864(-3,1)(0,0)(0,0)$\\
$(2,1)$&$\Z_2$&$-L(7,1)$&$L10n44(7,1)(0,0)$\\
$(1,2)$&$\Z_2$&$-L(3,1)$&$L12n784(3,1)(0,0)$\\
$(1,4)$&$\Z_2$&$L(7,4)$&$L13n4343(-7,4)(0,0)$\\
\midrule
\multicolumn{4}{l}{\textbf{$o9\_40363$}}\\
\midrule
$(0,1)$&SFS &$S^3$&$K12n479$ (QA)\\
\multirow{3}{*}{$(1,1)$}&\multirow{3}{*}{\begin{tabular}{@{}>{\raggedright}p{0.8\linewidth}>{\raggedleft}p{0.1\linewidth}}$\Z_2^2$ & {\Vast\{} \end{tabular}}
&$-L(2,1)$& $L14n43377(0,0)(2,1)$ \\ 
&&$-L(3,1)$&$L13n9451(0,0)(3,1)(0,0)$\\ 
&&$-L(7,1)$&$L11n419(0,0)(7,1)(0,0)$\\
$(-1,1)$&$\Z_2$&$S^3$&$L12n785$ (QA)\\
$(-1,2)$&$\Z_2$&$L(3,1)$&$L11a225(-3,1)(0,0)$\\
$(-1,3)$&$\Z_2$&$-L(4,1)$&$L12a1861(0,0)(4,1)(0,0)$\\
$(-3,4)$&$\Z_2$&$L(7,4)$&$L13n4363(-7,4)(0,0)$\\

\bottomrule
\end{tabular}
}
\end{table}
			
\begin{table}[htbp]
{\small
\caption{Table~\ref{tab:symInfo1} continued.}
\label{tab:symInfo2}
\begin{tabular}{@{}
>{\raggedright}p{0.1\linewidth}
>{\raggedright}p{0.15\linewidth}	
>{\raggedright}p{0.21\linewidth}
>{\raggedright\arraybackslash}p{0.31\linewidth}@{}}
\toprule
slope  & symmetry &  quotient manifold& branching set \\
\midrule
\multicolumn{4}{l}{\textbf{$o9\_40487$}}\\
\midrule
$(0,1)$&$\Z_2$&$S^3$&$L11n152$ (QA)\\
$(1,1)$&$\Z_2$&$S^3$&$K11n147$ (QA)\\
$(-1,1)$&$\Z_2$&$L(3,1)$&$L12n722(-3,1)(0,0)$\\
$(1,2)$&$\Z_2$&$L(3,1)$&$L12n968(-3,1)(0,0)$\\
$(2,1)$&$\Z_2$ &$L(2,1)$&$L11n333(0,0)(2,1)(0,0)$\\
$(3,1)$&$\Z_2$&$-L(7,1)$&$L12n1314(7,1)(0,0)$\\
$(-2,1)$&$\Z_2$&$L(2,1)$&$L12n1041(-2,1)(0,0)$\\
\midrule
\multicolumn{4}{l}{\textbf{$o9\_40504$}}\\
\midrule
$(0,1)$&$\Z_2$&$S^3$&$L11n179$ (QA)\\
$(1,1)$&$\Z_2$&$S^3$&$K11n166$ (QA)\\
$(-1,1)$&$\Z_2$&$L(3,1)$&$L12n715(-3,1)(0,0)$\\
$(-2,1)$&$\Z_2$&$L(2,1)$&$L12n1037(-2,1)(0,0)$\\
\multirow{3}{*}{$(2,1)$}&
\multirow{3}{*}{\begin{tabular}{@{}>{\raggedright}p{0.8\linewidth}>{\raggedleft}p{0.1\linewidth}}		$\Z_2^2$ & {\Vast\{} \end{tabular}}
& $L(3,1)$& $L13n7421(0,0)(-3,1)(0,0)$  \\ 
& 	&$-L(5,1)$&$L13n8521(0,0)(5,1)(0,0)$\\ 
& 	&$L(2,1)$&$L10a146(0,0)(-2,1)(0,0)$\\
\midrule
\multicolumn{4}{l}{\textbf{$o9\_40582$}}\\
\midrule
$(0,1)$&$\Z_2$&$S^3$&$L13n4413$ (QA)\\
$(1,1)$&$\Z_2$&$S^3$&$K13n2958$ (QA)\\
$(-1,1)$&$\Z_2$&$L(5,3)$&$L10n44(-5,3)(0,0)$\\
$(1,2)$&$\Z_2$&$-L(3,1)$&$L12n996(3,1)(0,0)$\\
$(2,1)$&$\Z_2$&$-L(3,1)$&$L13n9833(3,1)(0,0)(0,0)$\\
$(4,1)$&$\Z_2$&$L(5,1)$&$L11n350(0,0)(-5,1)(0,0)$\\
\midrule	
\multicolumn{4}{l}{\textbf{$o9\_42675$}}\\
\midrule
$(0,1)$&$\Z_2$&$S^3$&$L12n702$ (QA)\\
$(1,1)$&$\Z_2$&$-L(5,3)$&$L14n24428(5,3)(0,0)$\\
$(-1,1)$&$\Z_2$&$S^3$&$K12n730$ (QA)\\
$(2,1)$&$\Z_2$&$L(2,1)$&$L13n7552(0,0)(-2,1)(0,0)$\\
\multirow{3}{*}{$(-2,1)$}&
\multirow{3}{*}{\begin{tabular}{@{}>{\raggedright}p{0.8\linewidth}>{\raggedleft}p{0.1\linewidth}}		$\Z_2^2$ & {\Vast\{} \end{tabular}}
& $L(2,1)$
&$L11n223(-2,1)(0,0)$ \\ 
& 	&$-L(6,1)$&$L10n24(0,0)(6,1)$\\ 
& &$-L(8,3)$&$L13n9832(0,0)(8,3)(0,0)$\\
\bottomrule
\end{tabular}
}
\end{table}
		
\begin{proof} 
Let $K$ be an asymmetric census knot from $\mathcal A$. We start by classifying the dbc slopes of $K$ using Strategy~\ref{strat:dbc_slopes}.  

From Theorem~\ref{thm:exceptionalsym} we see that we have three different cases to consider: $K(r)$ has symmetry group $\Z_2$, $K(r)$ has symmetry group $\Z_2^2$, or $K(r)$ is a small Seifert fibered space.
				
In the first two cases, we perform the branch set search (Strategy~\ref{strat:branchlocussearch}) to search for an explicit branching set if the symmetry group is $\Z_2$ and for three explicit branching sets if the symmetry group is $\Z_2^2$.  In Tables~\ref{tab:symInfo1} and~\ref{tab:symInfo2} we present the simplest surgery description of the branching set that we found via this search. The symmetry exceptional slope is a dbc slope if and only if the quotient manifold for one of these symmetries is $S^3$.
				
If $K(r)$ is a small Seifert fibered space, then $K(r)$ is the double branched cover of a corresponding Montesinos link in $S^3$. Thus the three Seifert fibered surgery slopes are dbc slopes. This finishes the classification of the dbc slopes of the knots in $\mathcal A$.
				
To classify the set of quasi-alternating surgeries of the knots in $\mathcal A$ we use Strategy~\ref{strat:checkiflinkisQA} to show that all dbc slopes are quasi-alternating slopes. 
				
Finally, we show that no knot in $\mathcal A$ admits an alternating surgery. This fact was established by other techniques in \cite{BKM_alt}. We can immediately deduce that none of the symmetry-exceptional slopes are alternating using the uniqueness of the branching sets given by Lemma~\ref{lem:dbc_lemma}. Since each of the links in Tables~\ref{tab:symInfo1} and~\ref{tab:symInfo2} are non-alternating as evident from the \textit{n} in their DT name, it only remains to argue that none of the three small Seifert fibered slopes is alternating. In general, small Seifert fibered spaces may arise as the double branched cover of several different links in $S^3$. However, Motegi~\cite[Theorem 3.1]{Mo17} applies to show that these specific Seifert fibered spaces arise only as the double branched covers of a unique link, which is the Montesinos link associated to the Seifert invariants of the manifold.
\end{proof}
					
\section{The symmetric census L-space knots}\label{sec:sym_census}
Having classified the quasi-alternating slopes of the asymmetric census knots in Theorem~\ref{thm:asym_QA},  we next ask which symmetric census L-space knots admit a quasi-alternating surgery. All these knots are strongly invertible and thus every slope is a dbc slope.

\subsection{The classification of the census L-space knots} \label{sec:classification}
Amongst the knot exteriors in the SnapPy census, there are exactly two knots where the L-space status is not decided~\cite{Du19}. These are the knots $o9\_30150$ and $o9\_31440$. 
Here we confirm that they are L-space knots by showing that they admit quasi-alternating surgeries.

\begin{thm}\label{thm:2uncertain}
The knots $o9\_30150$ and $o9\_31440$ both have a quasi-alternating surgery. Hence they are L-space knots. 
\end{thm}
 
\begin{proof}
We identify a quasi-alternating slope using Strategy~\ref{strat:singleQA_slope}. For that, we first use Strategy~\ref{strat:branchlocussearch} to demonstrate that the $(-1,1)$-fillings on both of these manifolds are double branched coverings of small crossing knots:  
\[o9\_30150(-1,1) = \operatorname{DBC}(K13n3009) \,\,\, \mbox{and} \,\,\, o9\_31440(-1,1)= \operatorname{DBC}(K13n2028).\]  
To complete the proof, we use Strategy~\ref{strat:checkiflinkisQA} to show that both knots $K13n3009$ and $K13n2028$ are quasi-alternating.
 
Figure~\ref{fig:K13n3009-K13n2028}(Left) shows a diagram of the knot $K13n3009$ with a crossing circled. This knot has determinant $87$. One smoothing of the knot gives the knot $K12n666$ which is quasi-alternating by Jablan's table~\cite{Ja14} and has determinant $81$. The other smoothing gives the connected sum of a trefoil and a Hopf link which is alternating and has determinant $6$. Since $87=81+6$ and both smoothings are quasi-alternating, the knot $K13n3009$ is quasi-alternating.
 
Figure~\ref{fig:K13n3009-K13n2028}(Right) shows a diagram of the knot $K13n2028$ with a crossing circled. This knot has determinant $105$. One smoothing of the knot gives the knot $K12n598$ which is quasi-alternating by Jablan's table~\cite{Ja14} and has determinant $99$. The other smoothing gives the connected sum of a trefoil and a Hopf link which is alternating and has determinant $6$.  Since $105=99+6$ and both smoothings are quasi-alternating, the knot $K13n2028$ is quasi-alternating.
 \end{proof}

 \begin{figure}
     \centering
     \includegraphics[width=.46\textwidth]{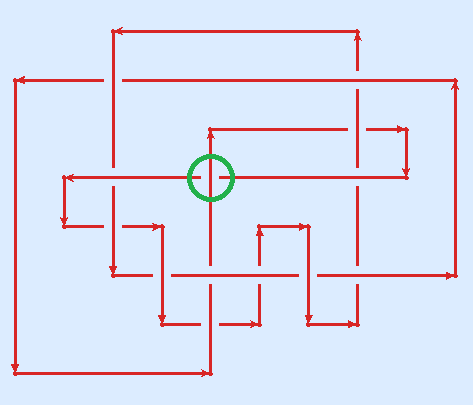} \qquad \includegraphics[width=.4\textwidth]{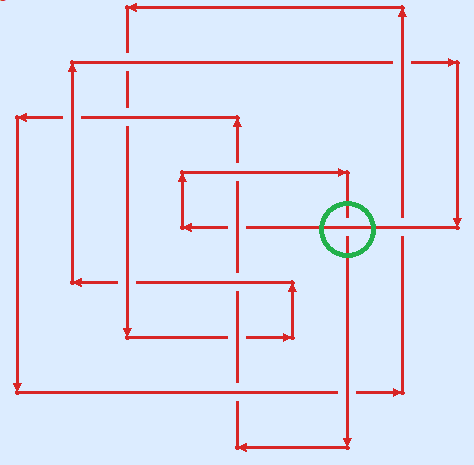} 
     \caption{(Left) A diagram of the knot $K13n3009$ with a quasi-alternating crossing circled. (Right) A diagram of the knot $K13n2028$ with a quasi-alternating crossing circled. Both diagrams were created using SnapPy~\cite{CDGW}.}
     \label{fig:K13n3009-K13n2028}
 \end{figure}

\subsection{Existence of quasi-alternating surgeries}\label{sec:QA_existence}
Here our main results are summarized in the following theorem. 
			
\begin{thm} \label{thm:QAsurgeries}
Every census L-space knot admits a quasi-alternating surgery, except for $t09847$ and $o9\_30634$. 
				
More precisely:
\begin{itemize}
	\item $431$ of the $632$ SnapPy census L-space knots admit a Seifert fibered surgery.  Some admit multiple Seifert fibered surgeries giving a total of $665$ Seifert fibered surgeries.
	\item These Seifert fibered surgeries are all quasi-alternating surgeries, except for the $8$ surgeries displayed in Table~\ref{tab:nonQA}.  
    \item Every census knot except for $t09847$ and $o9\_30634$ admits a quasi-alternating surgery. A quasi-alternating surgery for each census L-space knot without an alternating surgery or a quasi-alternating Seifert fibered space surgery is displayed in Table~\ref{tab:QATable}. (The alternating surgeries can be found in~\cite{BKM_alt}.)
\end{itemize}
\end{thm}
			
\begin{table}[htbp]
{\tiny
\caption{We display a quasi-alternating surgery for every census L-space knot that does not have an alternating surgery or a quasi-alternating Seifert fibered space surgery. The $17$ crossing knot $17nh\_0008151$ is in Burton's notation~\cite{Burton}.}
\label{tab:QATable}
\begin{tabular}
{@{}lll@{}c@{}lll@{}c@{}lll@{}}
\addlinespace[-\aboverulesep] 
\cmidrule[\heavyrulewidth]{1-3}  \cmidrule[\heavyrulewidth]{5-7}  \cmidrule[\heavyrulewidth]{9-11} 
knot  & slope  &  QA br set& \phantom{xxx} &knot  & slope &  QA br set& \phantom{xxx} &knot  & slope &  QA br set \\
\cmidrule{1-3}  \cmidrule{5-7}  \cmidrule{9-11}
	$s652$ & $(0, 1)$ & $L10n22$ &&
	$v1915$ & $(0, 1)$ & $K11n26$ &&
	$v2384$ & $(0, 1)$ & $L11n90$ \\
	$v2871$ & $(2, 1)$ & $K14n8212$ &&
	$v3105$ & $(0, 1)$ & $K10n8$ &&
	$v3335$ & $(0, 1)$ & $L10n52$ \\
	$v3482$ & $(0, 1)$ & $L11n37$ &&
	$t04557$ & $(0, 1)$ & $L12n761$ &&
	$t06573$ & $(0, 1)$ & $L12n373$ \\
	$t06715$ & $(0, 1)$ & $K11n25$ &&
	$t07104$ & $(1, 1)$ & $L12n1043$ &&
	$t08114$ & $(-1, 1)$ & $K11n142$ \\
	$t08532$ & $(0, 1)$ & $L12n674$ &&
	$t08936$ & $(0, 1)$ & $L11n157$ &&
	$t09126$ & $(0, 1)$ & $L12n737$ \\
	$t09284$ & $(0, 1)$ & $L12n8$ &&
	$t09882$ & $(-2, 1)$ & $K14n11995$ &&
	$t10215$ & $(0, 1)$ & $L11n73$ \\
	$t10292$ & $(0, 1)$ & $K11n68$ &&
	$t10496$ & $(0, 1)$ & $L12n83$ &&
	$t10832$ & $(0, 1)$ & $L11n231$ \\
	$t11198$ & $(0, 1)$ & $K11n5$ &&
	$t11376$ & $(1, 1)$ & $L11n77$ &&
	$t11548$ & $(0, 1)$ & $L12n82$ \\
	$t11887$ & $(0, 1)$ & $L12n923$ &&
	$t12288$ & $(0, 1)$ & $L12n338$ &&
	$t12533$ & $(0, 1)$ & $K12n407$ \\
	$o9\_10020$ & $(1, 1)$ & $L13n5889$ &&
	$o9\_16685$ & $(0, 1)$ & $L13n2087$ &&
	$o9\_17646$ & $(0, 1)$ & $K12n76$ \\
	$o9\_18341$ & $(-1, 1)$ & $K13n2479$ &&
	$o9\_18646$ & $(0, 1)$ & $L12n1066$ &&
	$o9\_19247$ & $(0, 1)$ & $L13n746$ \\
	$o9\_19364$ & $(-1, 1)$ & $L13n4475$ &&
	$o9\_24290$ & $(0, 1)$ & $L12n367$ &&
	$o9\_24401$ & $(0, 1)$ & $L11n130$ \\
	$o9\_24779$ & $(0, 1)$ & $K13n677$ &&
	$o9\_24946$ & $(-1, 1)$ & $L12n692$ &&
	$o9\_25110$ & $(0, 1)$ & $K12n81$ \\
	$o9\_25199$ & $(0, 1)$ & $K12n108$ &&
	$o9\_25444$ & $(0, 1)$ & $L13n2742$ &&
	$o9\_26471$ & $(-1, 1)$ & $L13n3693$ \\
	$o9\_26570$ & $(0, 1)$ & $L12n407$ &&
	$o9\_28284$ & $(0, 1)$ & $L12n753$ &&
	$o9\_29048$ & $(-1, 1)$ & $K12n311$ \\
	$o9\_29551$ & $(1, 1)$ & $K13n3186$ &&
	$o9\_29648$ & $(0, 1)$ & $K12n183$ &&
	$o9\_29751$ & $(0, 1)$ & $K12n290$ \\
	$o9\_29766$ & $(-1, 1)$ & $K14n9043$ &&
	$o9\_30142$ & $(-1, 1)$ & $K13n668$ &&
	$o9\_30150$ & $(0, 1)$ & $L12n784$ \\
	$o9\_31267$ & $(0, 1)$ & $K12n14$ &&
	$o9\_31321$ & $(0, 1)$ & $K12n327$ &&
	$o9\_31440$ & $(0, 1)$ & $L12n782$ \\
	$o9\_31481$ & $(0, 1)$ & $K12n70$ &&
	$o9\_32044$ & $(-3, 1)$ & $K15n42102$ &&
	$o9\_32065$ & $(0, 1)$ & $L12n975$ \\
	$o9\_32150$ & $(-1, 1)$ & $K13n2177$ &&
	$o9\_32314$ & $(0, 1)$ & $L12n340$ &&
	$o9\_32471$ & $(0, 1)$ & $K12n112$ \\
	$o9\_32964$ & $(-4, 1)$ & $17nh\_0008151$ &&
	$o9\_33189$ & $(1, 1)$ & $L12n321$ &&
	$o9\_33284$ & $(0, 1)$ & $K13n668$ \\
	$o9\_33380$ & $(0, 1)$ & $K13n587$ &&
	$o9\_33430$ & $(0, 1)$ & $K13n3880$ &&
	$o9\_33486$ & $(0, 1)$ & $L13n406$ \\
	$o9\_33801$ & $(0, 1)$ & $L12n686$ &&
	$o9\_33944$ & $(0, 1)$ & $L12n353$ &&
	$o9\_33959$ & $(0, 1)$ & $L12n874$ \\
	$o9\_34000$ & $(0, 1)$ & $L12n403$ &&
	$o9\_34409$ & $(0, 1)$ & $K13n602$ &&
	$o9\_34689$ & $(-1, 1)$ & $L12n1086$ \\
	$o9\_35666$ & $(0, 1)$ & $K12n17$ &&
	$o9\_35720$ & $(0, 1)$ & $K12n395$ &&
	$o9\_35928$ & $(0, 1)$ & $K12n85$ \\
	$o9\_36114$ & $(0, 1)$ & $K12n197$ &&
	$o9\_36250$ & $(-1, 1)$ & $K12n401$ &&
	$o9\_36380$ & $(0, 1)$ & $L13n1195$ \\
	$o9\_36544$ & $(0, 1)$ & $L11n177$ &&
	$o9\_36809$ & $(0, 1)$ & $K11n123$ &&
	$o9\_36958$ & $(0, 1)$ & $K12n73$ \\
	$o9\_37050$ & $(0, 1)$ & $L13n1166$ &&
	$o9\_37291$ & $(0, 1)$ & $L12n975$ &&
	$o9\_37482$ & $(0, 1)$ & $L12n855$ \\
	$o9\_37551$ & $(0, 1)$ & $K12n280$ &&
	$o9\_37751$ & $(0, 1)$ & $L12n754$ &&
	$o9\_37851$ & $(0, 1)$ & $L12n823$ \\
	$o9\_38287$ & $(0, 1)$ & $L12n801$ &&
	$o9\_38679$ & $(0, 1)$ & $L13n2013$ &&
	$o9\_38811$ & $(0, 1)$ & $K11n154$ \\
	$o9\_38928$ & $(0, 1)$ & $L11n178$ &&
	$o9\_38989$ & $(0, 1)$ & $L12n391$ &&
	$o9\_39521$ & $(0, 1)$ & $L13n1206$ \\
	$o9\_39606$ & $(-1, 1)$ & $K12n139$ &&
	$o9\_39859$ & $(1, 1)$ & $L12n923$ &&
	$o9\_39879$ & $(-1, 1)$ & $K12n593$ \\
	$o9\_39981$ & $(0, 1)$ & $L12n362$ &&
	$o9\_40026$ & $(0, 1)$ & $L13n3152$ &&
	$o9\_40075$ & $(0, 1)$ & $K12n26$ \\
	$o9\_40487$ & $(0, 1)$ & $L11n152$ &&
	$o9\_40504$ & $(0, 1)$ & $L11n179$ &&
	$o9\_40582$ & $(0, 1)$ & $L13n4413$ \\
	$o9\_42224$ & $(0, 1)$ & $K12n2$ &&
	$o9\_42493$ & $(0, 1)$ & $K12n536$ &&
	$o9\_42675$ & $(0, 1)$ & $L12n702$ \\
	$o9\_42961$ & $(0, 1)$ & $L13n4875$ &&
	$o9\_43750$ & $(0, 1)$ & $K13n4022$ &&
	$o9\_43857$ & $(0, 1)$ & $L13n3470$ \\
	\cmidrule{1-3}  \cmidrule{5-7}  \cmidrule{9-11}
	\end{tabular}
}
\end{table}
			
\begin{proof}
	For the first two bullets we load from Dunfield's list of exceptional surgeries~\cite{Du18} all Seifert fibered space surgeries and use Issa's classification of quasi-alternating Montesinos links~\cite{Is18} to verify the statements. 
				
	For the last bullet, we use Strategy~\ref{strat:singleQA_slope} to find an explicit quasi-alternating branching set. This identifies quasi-alternating slopes for all but $6$ knots.	
				
    For the remaining $6$ knots, we used Strategy~\ref{strat:tangle_ext} to create the tangle exteriors. Then we can fill the tangle exteriors to create explicit branching sets of surgeries and check with Strategy~\ref{strat:checkiflinkisQA} if this branching set is quasi-alternating. This identifies quasi-alternating branching sets for $4$ more knots. 
	
    The remaining two knots are $t09847$ and $o9\_30634$ for which this search did not find a quasi-alternating branching set. In fact, these two knots admit no Khovanov thin slope and hence no quasi-alternating slope as shown in~\cite{BKM_thin}.
\end{proof}
			
\begin{table}[htbp] 
\caption{The $8$ SFS surgeries along census L-space knots where the corresponding Montesinos links are not quasi-alternating. Here we are using the Regina notation for Seifert fibered spaces~\cite{BBP+}. Refer to Section~\ref{sec:QA_torus} for more details.}
\label{tab:nonQA}
\begin{tabular}{cccc}
	\toprule
	Knot	& slope & SFS & Montesinos link\\
	\midrule
	$v2871$ & $(0, 1)$ & $SFS [S2: (2,1) (5,2) (7,-4)]$& $K11n20$ \\
	$v2871$&$(-1, 1)$&$SFS [S2: (2,1) (5,2) (8,-5)$&$L11n127$\\
	$t09847$&$(0, 1)$&$SFS [S2: (2,1) (7,2) (8,-5)]$&$L12n612$\\
	$t09882$&$(0, 1)$&$SFS [S2: (3,1) (5,3) (7,-5)]$&$K12n322$\\
	$o9\_30634$&$(0, 1)$&$SFS [S2: (2,1) (8,3) (9,-7)]$&$L13n2575$\\
	$o9\_32044$&$(0, 1)$&$SFS [S2: (2,1) (5,2) (12,-7)]$&$L12n799$\\
	$o9\_32964$&$(0, 1)$&$SFS [S2: (2,1) (7,3) (11,-7)]$&$K13n2557$\\
	$o9\_32964$&$(-1, 1)$&$SFS [S2: (2,1) (8,3) (9,-5)]$&$L13n2596$\\
	\bottomrule
\end{tabular}
\end{table}
			
\section{The Baker--Luecke asymmetric L-space knots}\label{sec:Baker_Luecke}

In~\cite{BL17} an infinite family of interesting hyperbolic asymmetric L-space knots was constructed and in~\cite{BL17,BKM_alt} their alternating slopes were studied. 
The subfamily of knots $K_{(m,b1,a1,a2,a3)}$ defined in~\cite[\S 11.4]{BL17} for non-negative integers $a_3, a_2, a_1, m, b_1$ gives a small collection of examples. In~\cite{BKM_alt} we have shown that the $14$ simplest of these knots all have exactly two alternating slopes, and these pairs of slopes are consecutive integers. Here we show that $6$ of these have no further dbc slopes and give evidence that the same holds true for the other $8$ knots. 

\begin{thm} \label{thm:BLknots}\hfill
\begin{enumerate}
    \item $K_{(m,b1,a1,a2,a3)}$ is an asymmetric hyperbolic L-space knot for the $14$ values of $(m,b1,a1,a2,a3)$ shown in Table~\ref{tab:BL_Knots_alt}. 
    \item All their symmetry-exceptional slopes are presented in Table~\ref{tab:BL_Knots_alt}. 
    \item For $6$ of the above $14$ knots there are exactly two dbc surgeries. Both dbc surgeries are consecutive integers and alternating. These surgeries together with their branching sets are displayed in Table~\ref{tab:BL_Knots_alt}. 
    \item For the remaining $8$ of the above $14$ knots there are still two dbc surgeries. Both these dbc surgeries are consecutive integers and alternating. These surgeries together with their branching sets are displayed in Table~\ref{tab:BL_Knots_alt}. Moreover, for each of these $8$ knots there are at most two more slopes that might be dbc slopes. These further possible dbc slopes are present in Table~\ref{tab:BL_Knots_alt}. All other slopes are proven to not be dbc slopes. 
\end{enumerate}
\end{thm}

\begin{table}[htbp]
{\scriptsize
\caption{The first $14$ asymmetric L-space knots from~\cite{BL17}, their symmetry-exceptional slopes, and their branching sets and quotient manifolds (if known). We observe that every knot has exactly two dbc slopes that are both alternating and consecutive integers. The slopes that are marked with \textit{unknown} branching set are just conjectured (but not proven) to be not dbc slopes.} 
\label{tab:BL_Knots_alt}
\begin{tabular}{@{}
>{\raggedright}p{0.155\linewidth}
>{\raggedright}p{0.095\linewidth}
>{\raggedright}p{0.09\linewidth}	
>{\raggedright}p{0.13\linewidth}
>{\raggedright\arraybackslash}p{0.37\linewidth}@{}}
\toprule
$(m,b_1,a_1,a_2,a_3)$ & slope  & symmetry &  quotient& branching set \\
\midrule
$(1, 1, 1, 1, 0)$   &$(271,1)$ & $\Z_2$ & $S^3$& $K12a402$\\
 &$(272,1)$ & $\Z_2$   & $S^3$& $L12a955$\\
\multirow{3}{*}{}&\multirow{3}{*}{$(273, 1)$}&
\multirow{3}{*}{\begin{tabular}{@{}>{\raggedright}p{0.8\linewidth}>{\raggedleft}p{0.1\linewidth}}		$\Z_2^2$ & {\Vast\{} \end{tabular}}
&$-L(3,1)$& $L13n3686(3, 1)(0,0)$  \\ 
&& 	 &unknown&unknown\\ 
&& 	&unknown&unknown\\
&$(543, 2)$ & $\Z_2$ & $-L(3,1)$&$L13a2543(3, 1)(0,0)$\\
&$(2171, 8)$ & $\Z_2$ & unknown&unknown\\
\hline
$(1, 1, 0, 1, 1)$       &$(469, 1)$ & $\Z_2$& unknown & unknown\\
&$(470,1)$ & $\Z_2$ & $S^3$& $L13a1826$ \\
&$(471,1)$ & $\Z_2$   & $S^3$&$K13a4669$\\
&$(472, 1)$ & $\Z_2$ & $-L(4,1)$& $L14n55138(0,0)(4, 1)(0,0)$\\
&$(1411, 3)$ & $\Z_2$ & unknown& unknown\\
\hline                 
$(1, 1, 1, 2, 0)$       &$(415,1)$ & $\Z_2$ & $S^3$&$K13a1838$  \\
&$(416,1)$ & $\Z_2$   & $S^3$&$L13a3060$\\
&$(417, 1)$ & $\Z_2$ & $-L(3,1)$&$L14n16219(3, 1)(0,0)$ \\
&$(831, 2)$ & $\Z_2$ & $L(3,1)$& $L14a9663(-3, 1)(0,0)$\\
\hline                  
$(1, 1, 2, 1, 0)]$      &$(556,1)$ & $\Z_2$ & $S^3$& $L14a13430$ \\
&$(557,1)$ & $\Z_2$   & $S^3$& $K14a14150$\\
&$(558, 1)$ & $\Z_2$ & unknown& unknown\\
&$(6120, 11)$ & $\Z_2$ & unknown& unknown\\
 \hline                 
$(1, 2, 1, 1, 0)$       &$(554,1)$ & $\Z_2$ & $S^3$&  $L14a6302$\\
&$(555,1)$ & $\Z_2$   & $S^3$&$K14a12040$\\
  \hline                
$(2, 1, 1, 1, 0)$       &$(587,1)$ & $\Z_2$ & $S^3$& $K14a5753$  \\
&$(588,1)$ & $\Z_2$   & $S^3$&$L14a12460$\\
 &$(589, 1)$ & $\Z_2$ & unknown& unknown\\
&$(6460, 11)$ & $\Z_2$ & unknown& unknown\\
   \hline               
$(1, 1, 1, 1, 1)$       &$(1155,1)$ & $\Z_2$ & $S^3$& $K15a48589$ \\
 & $(1156,1)$ & $\Z_2 $ & $S^3$&$\text{15-crossing alternating link \cite{BL17}}$\\
&$(5778, 5)$ & $\Z_2$& unknown &unknown \\
   \hline               
$(1, 1, 0, 2, 1)$       &$(1008, 1)$ & $\Z_2$ & unknown&unknown \\
&$(1009,1)$ & $\Z_2$ & $S^3$& $K15a63354$ \\
& $(1010,1)$ & $\Z_2 $ & $S^3$& $\text{15-crossing alternating link \cite{BL17}}$\\
  \hline                
$(1, 2, 0, 1, 1)$ &$(965,1)$ & $\Z_2$ & $S^3$&  $K15a80635$\\
& $(966,1)$ & $\Z_2$ & $S^3$& $\text{15-crossing alternating link \cite{BL17}}$\\
\hline               
$(1, 1, 2, 2, 0)$ &$(845,1)$ & $\Z_2$ & $S^3$& $K15a71795$\\
&$(846,1)$ & $\Z_2$   & $S^3$&$\text{15-crossing alternating link \cite{BL17}}$\\
\hline             
$(1, 2, 1, 2, 0)$ &$(843,1)$ & $\Z_2$ & $S^3$&  $K15a27596$\\
 & $(844,1)$ & $\Z_2$  & $S^3$& $\text{15-crossing alternating link \cite{BL17}}$\\
\hline            
$(2, 1, 1, 2, 0)$ &$(911,1)$ & $\Z_2$ & $S^3$& $K15a50514$ \\
&$(912,1)$ & $\Z_2$   & $S^3$& $ \text{15-crossing alternating link \cite{BL17}}$ \\
\hline           
$(1, 1, 0, 1, 2)$ &$(1142,1)$ & $\Z_2 $& $S^3$& $\text{15-crossing alternating link \cite{BKM_alt}}$ \\
&$(1143,1)$ & $\Z_2$   & $S^3$& $K15a43818$\\
&$(1144, 1)$ & $\Z_2$ & unknown&unknown \\
\hline           
$(2, 1, 0, 1, 1)$ &$(1066,1)$ & $\Z_2 $& $S^3$& $\text{15-crossing alternating link \cite{BKM_alt}}$ \\
&$(1067,1)$ & $\Z_2$  & $S^3$ &$K15a84691$\\
&$(1068, 1)$ & $\Z_2$ & unknown& unknown\\     
\bottomrule
\end{tabular}
}
\end{table}

\begin{proof} As in~\cite{BKM_alt}, we load the surgery description of the $K_{(m,b1,a1,a2,a3)}$ given in~\cite{BL17} into SnapPy from which we can build the $14$ examples shown in Table~\ref{tab:BL_Knots_alt}. Then the verified functions in SnapPy tell us that these knots are hyperbolic with trivial symmetry group. That they are L-space knots was shown in~\cite{BL17} by demonstrating that they have alternating surgeries. 

Now let $K$ be one of the above $14$ knots. We then use Strategy~\ref{strat:symmetryexceptional_slopes} to compute the symmetry-exceptional slopes of $K$. In total, we have found $43$ slopes with symmetry groups $\Z_2$ and $1$ slope with symmetry group $\Z_2^2$.  These are displayed in Table~\ref{tab:BL_Knots_alt}.
Every dbc slope is contained in this list. 

Then we use the branch set search (Strategy~\ref{strat:branchlocussearch}) to look for the branching sets of the symmetry-exceptional slopes. Every one of these knots has exactly two alternating slopes; these were already found in~\cite{BKM_alt}. 
Moreover, we have identified explicit branching sets for $5$ more slopes among these knots. These other $5$ branching sets are given by links in manifolds different from $S^3$ and thus do not certify these slopes as dbc slopes. Since $4$ of these $5$ slopes have symmetry $\Z_2$, this is sufficient to conclude that those slopes are not dbc slopes.  The remaining slope is the one with symmetry $\Z_2^2$ and its other branching sets have not yet been determined. 
For the other symmetry-exceptional slopes our code could not identify explicit branching sets. These are marked as `unknown' in the table.
\end{proof}

\begin{rem}
    The fact that the branch set search did not terminate for some of these examples is because the branching sets are too complicated to appear in our search. We expect that these slopes are not dbc slopes, that is, we expect the quotient manifold is not $S^3$. However as the search was inconclusive, their status as dbc slopes remains unknown in Table~\ref{tab:BL_Knots_alt}.
\end{rem}

\section{The quasi-alternating slopes of torus knots}\label{sec:QA_torus}
In the previous sections, we dealt with finite
families of hyperbolic knots. In this section, we address an infinite family of non-hyperbolic knots, the collection of torus knots.
We classify the quasi-alternating slopes of torus knots. For coprime integers $a>b>1$, we denote by $T_{a,b}$ the $(a,b)$-torus knot. 
		
\begin{thm}\label{thm:torus_knotsQA}
A slope $p/q\in\Q$ is quasi-alternating for the torus knot $T_{a,b}$ with $a>b>1$ if and only if
\begin{equation*}
\frac{p}{q}>ab-\max\left\{\frac{a}{m},\frac{b}{n}\right\},
\end{equation*}
where integers $m,n$ satisfy $bm+an=ab+1$ and $1\leq m<a$ and $1\leq n <b$. 
\end{thm}

Since a Dehn surgery on a torus knot is either a lens space, a connected sum of lens spaces, or a small Seifert fibered space~\cite{Moser1971elementary} and lens spaces are double branched covers of $2$-bridge links (which are alternating), we only need to understand which small Seifert fibered spaces are double branched covers of quasi-alternating links. For that, we appeal to Issa's work on quasi-alternating Montesinos links~\cite{Is18} and Moser's classification of surgeries on torus knots~\cite{Moser1971elementary}.
		
\begin{thm} [Moser~\cite{Moser1971elementary}]\label{thm:Moser}
	The $(p/q)$-surgery on $T_{a,b}$ is diffeomorphic to the manifold 
	\begin{align*}
		SFS[S2:(a,d)(b,c)(abq-p,q)].
	\end{align*}
	with surgery diagram of Figure~\ref{fig:SFS} where $c$ and $d$ are integers such that $ac+bd=-1$.
	If $p/q \neq ab$, then this is a Seifert fibered space.
	If $p/q=ab$ then it is a generalized Seifert fibered space diffeomorphic to $L(a,b)\#L(b,a)$.
\end{thm}

Recall that we use Regina's notation for Seifert fibered spaces \cite{BBP+}. Since the notation and conventions used in \cite{Moser1971elementary} differ from ours, we briefly and roughly sketch a proof using rational Kirby calculus. Essentially the same proof can also be found in Lemma 4.4 of~\cite{Owens_Strle12}.
		
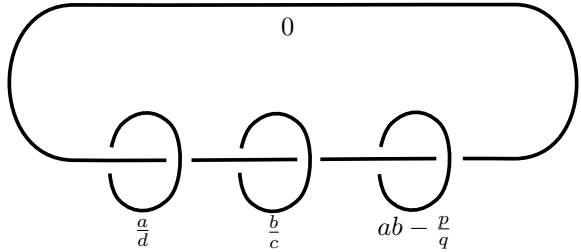
\begin{figure}[htbp] 
	\centering
	\def\svgwidth{0,6\columnwidth}
	%% Creator: Inkscape 1.1.2 (0a00cf5339, 2022-02-04), www.inkscape.org
%% Accompanies image file 'SFS.pdf' (pdf, eps, ps)
\begingroup%
  \makeatletter%
  \providecommand\color[2][]{%
    \errmessage{(Inkscape) Color is used for the text in Inkscape, but the package 'color.sty' is not loaded}%
    \renewcommand\color[2][]{}%
  }%
  \providecommand\transparent[1]{%
    \errmessage{(Inkscape) Transparency is used (non-zero) for the text in Inkscape, but the package 'transparent.sty' is not loaded}%
    \renewcommand\transparent[1]{}%
  }%
  \providecommand\rotatebox[2]{#2}%
  \newcommand*\fsize{\dimexpr\f@size pt\relax}%
  \newcommand*\lineheight[1]{\fontsize{\fsize}{#1\fsize}\selectfont}%
  \ifx\svgwidth\undefined%
    \setlength{\unitlength}{337.44700033bp}%
    \ifx\svgscale\undefined%
      \relax%
    \else%
      \setlength{\unitlength}{\unitlength * \real{\svgscale}}%
    \fi%
  \else%
    \setlength{\unitlength}{\svgwidth}%
  \fi%
  \global\let\svgwidth\undefined%
  \global\let\svgscale\undefined%
  \makeatother%
  \begin{picture}(1,0.42109724)%
    \lineheight{1}%
    \setlength\tabcolsep{0pt}%
    \put(0,0){\includegraphics[width=\unitlength,page=1]{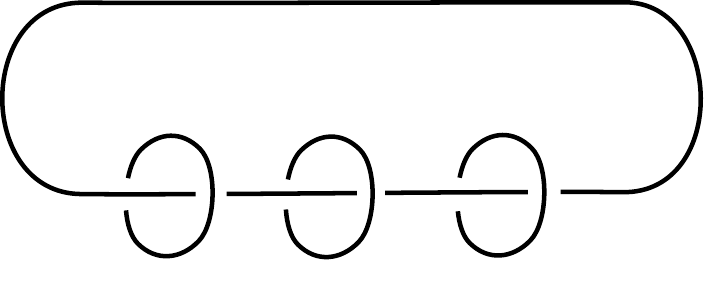}}%
    \put(0.22115315,0.01044094){\color[rgb]{0,0,0}\makebox(0,0)[lt]{\lineheight{0.1}\smash{\begin{tabular}[t]{l}$\frac{a}{d}$\end{tabular}}}}%
    \put(0.45283591,0.00737963){\color[rgb]{0,0,0}\makebox(0,0)[lt]{\lineheight{0.1}\smash{\begin{tabular}[t]{l}$\frac{b}{c}$\end{tabular}}}}%
    \put(0.64801262,0.01321451){\color[rgb]{0,0,0}\makebox(0,0)[lt]{\lineheight{0.1}\smash{\begin{tabular}[t]{l}$ab-\frac{p}{q}$\end{tabular}}}}%
    \put(0.47907139,0.36385091){\color[rgb]{0,0,0}\makebox(0,0)[lt]{\lineheight{0.1}\smash{\begin{tabular}[t]{l}$0$\end{tabular}}}}%
  \end{picture}%
\endgroup%

	\caption{A surgery diagram of the Seifert fibered space $SFS[S2:(a,d)(b,c)(abq-p,q)]$}
	\label{fig:SFS}
\end{figure}
		
\begin{proof}[Proof sketch of Theorem~\ref{thm:Moser}]
    Figure~\ref{fig:SFS} shows a surgery diagram for these manifolds.
	By a slam dunk on the knot with surgery coefficient $a/d$ into the $0$ framed component, we obtain a knot $K_0$ with surgery coefficient $-d/a$ linked by the components $K_1$ and $K$ with surgery coefficients $b/c$ and  $ab-p/q$ respectively. Together $K_0 \cup K_1$ form a Hopf link where $K$ is a curve in the Heegaard torus between them that is isotopic to the meridian of $K_0$ and the longitude of $K_1$.  Since $ac+bd=-1$, a sequence of Rolfsen twists on $K_0$ and $K_1$ informed from a continued fraction expansion of $-d/a$ with partial fraction $b/c$ reduces both the surgery coefficients of $K_0$ and $K_1$ to $1/0$, twists $K$ to the $(a,b)$ curve on the Heegaard torus, and reduces its surgery coefficient to $p/q$.
\end{proof}
	
In~\cite{Is18} Issa classifies quasi-alternating Montesinos knots as follows.
	
\begin{thm}[Issa~\cite{Is18}]\label{thm:QA_Montesinos}
	Let $L=M(e;t_1=\alpha_i/\beta_i,\ldots,t_n=\alpha_n/\beta_n)$ be a Montesinos link in standard form, i.e.\ such that $t_i>1$. Then $L$ is quasi-alternating if and only if
	\begin{enumerate}
		\item $e<1$, or
		\item $e=1$ and  $\frac{\beta_i}{\alpha_i}+\frac{\beta_j}{\alpha_j}>1$ for some $i\neq j$, or
		\item $e>n-1$, or
		\item $e=n-1$ and $\frac{\beta_i}{\alpha_i}+\frac{\beta_j}{\alpha_j}<1$ for some $i\neq j$.
	\end{enumerate}
\end{thm}
	
In addition, it is shown in~\cite[Corollary~1]{Is18} that a Seifert fibered space is the double branched cover of a quasi-alternating link if and only if it is the double branched cover of a quasi-alternating Montesinos link. The Seifert fibered space $SFS[S2:(\alpha_1,\beta_1)\ldots(\alpha_n,\beta_n)]$ in Regina's notation is the double branched cover of the Montesinos knot $M(0;\alpha_i/\beta_i,\ldots,\alpha_n/\beta_n)$ which is in general not in Issa's standard form.
		
\begin{proof}[Proof of Theorem~\ref{thm:torus_knotsQA}]
	First, since a quasi-alternating slope $p/q$ is an L-space slope and $T_{a,b}$ is a positive L-space knot, we may assume that $p\geq q >0$.
		
	Now by Theorem~\ref{thm:Moser},  $(p/q)$-surgery on $T_{a,b}$ is diffeomorphic to the (generalized) Seifert fibered space
	\begin{align*}
	SFS[S2:(a,d)(b,c)(abq-p,q)] 
	&= \dbc(M\left(e=0;\frac{a}{d}, \frac{b}{c}, \frac{abq-p}{q}\right)) \\
	&= \dbc(M\left(e=1;\frac{a}{d}, \frac{b}{b+c}, \frac{abq-p}{q}\right))    
	\end{align*}
	where $c$ and $d$ are any integers such that $ac+bd=-1$. Since $a>b>0$ it follows that $d>0>c$. 
	Furthermore the integers $c,d$ satisfying $ac+bd=-1$ may be chosen so that $a>d>0$ and $b>-c$, and hence $\frac{a}{d}>1$ and $\frac{b}{b+c}>1$. 
  
	Also, notice that if $p/q=ab-1$ then the above manifold is a lens space, and thus the slope $ab-1$ is quasi-alternating. From~\cite{Wa11} it follows that any slope $p/q\geq ab-1$ is quasi-alternating. Otherwise, we have $p/q<ab-1$ so that $\frac{abq-p}{q}>1$ and 
	Theorem~\ref{thm:QA_Montesinos} informs us that the Montesinos link
	$M\left(1;a/d, b/(b+c), ab-p/q\right)$
	is quasi-alternating if and only if one of the following is fulfilled:
	\begin{enumerate}
		\item $\frac{d}{a}+\frac{b+c}{b} >1$
		%    $d/a+1+c/b>1$    $d/a+c/b>0$  $cd+ab>0$
		\item $\frac{b+c}{b} + \frac{q}{abq-p}>1$     
		\item $\frac{d}{a} + \frac{q}{abq-p}>1$
	\end{enumerate}
	(1) cannot occur because $ad+bc=-1$.  Either (2) or (3) occurs if and only if either $\frac{q}{abq-p}>\frac{-c}{b}$ or  $\frac{q}{abq-p}>\frac{a-d}{a}$.   Since $abq-p>q>0$, $a>d>0$, and $b>-c>0$, this is occurs if and only if $\frac{p}{q} > ab-\max\{\frac{a}{a-d},\frac{b}{-c}\}$. 
	Thus, with \cite[Corollary~1]{Is18},  this gives the claimed classification of quasi-alternating slopes once we set $m=a-d$ and $n=-c$.
\end{proof}
\begin{rem}
For a knot $K$ in the 3-sphere, Owens and Strle \cite{Owens_Strle12} define $m(K)$ to be the infimum of all $p/q> 0$ such that $K(p/q)$ is the boundary of a smooth negative definite 4-manifold. It turns out that the lower bound in Theorem~\ref{thm:torus_knotsQA} coincides with $m(T_{a,b})$ with $a>b>1$. That is $p/q\in\Q$ is a quasi-alternating slope for $T_{a,b}$ if and only if $p/q>m(T_{a,b})$. Since every quasi-alternating surgery is the boundary of a smooth negative-definite 4-manifold, we have the inequality $m(T_{a,b})\leq ab-\max\left\{\frac{a}{m},\frac{b}{n}\right\}$. Equality readily follows by comparison with the formula for $m(T_{a,b})$ given in  \cite[Theorem~2]{Owens_Strle12}.
\end{rem}

\section{Formal L-space surgeries}\label{sec:FLSsurgeries}
In this section, we prove the results about formal L-space surgeries. As discussed in Remark~\ref{rem:formal_Lspace_QA} formal L-spaces can be seen as the $3$-manifold analogues of quasi-alternating links. First, we show that if the set of formal L-space surgery slopes is non-empty, then it is necessarily infinite.

\begin{thm}\label{thm:plusone}
    If $K(r)$ is a formal L-space for some $r>0$, then $K(r +n)$ is a formal L-space for every $n\in \Z_{\geq 0}$.
\end{thm}

\begin{proof}
    Let $K$ be a knot in $S^3$ and let $r=p/q\in \Q$ be a slope with $q\geq 1$. As illustrated in Figure~\ref{fig:triad1}, the manifold $K(p/q)$ is contained in a triad consisting of $Y_0=U(q/t)$, $Y_1=K(p/q)$, and $Y=K((p+q)/q)$, where $p/q=n-t/q$ with $n\in \Z$ and an integer $t$ with $q>t\geq 0$. Recall that $K(p/q)$ has first homology isomorphic to $\Z_p$ and thus the triad $(Y_0,Y_1,Y)$ fulfills the homology condition in the definition of a formal L-space. Since every lens space is a formal L-space, we see that $K((p+q)/q)=K(p/q + 1)$ is a formal L-space, whenever $K(p/q)$ is a formal L-space and $p/q>0$. Thus the theorem follows by induction.
\end{proof}

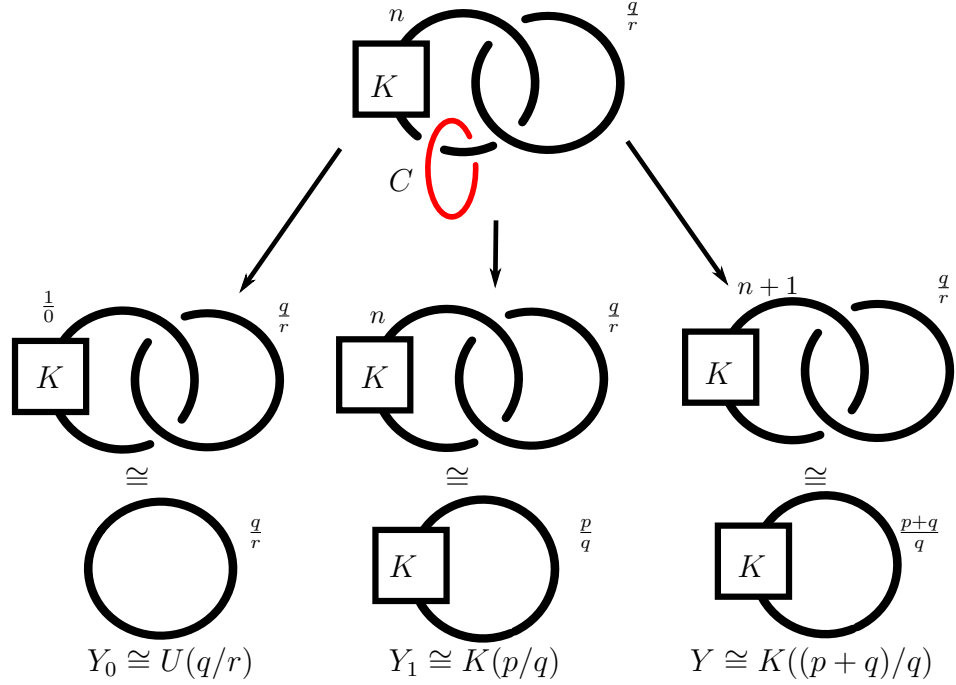
\begin{figure}[htbp]
\vspace{1cm}
  \centerline{
    \begin{overpic}[width=0.99\textwidth]{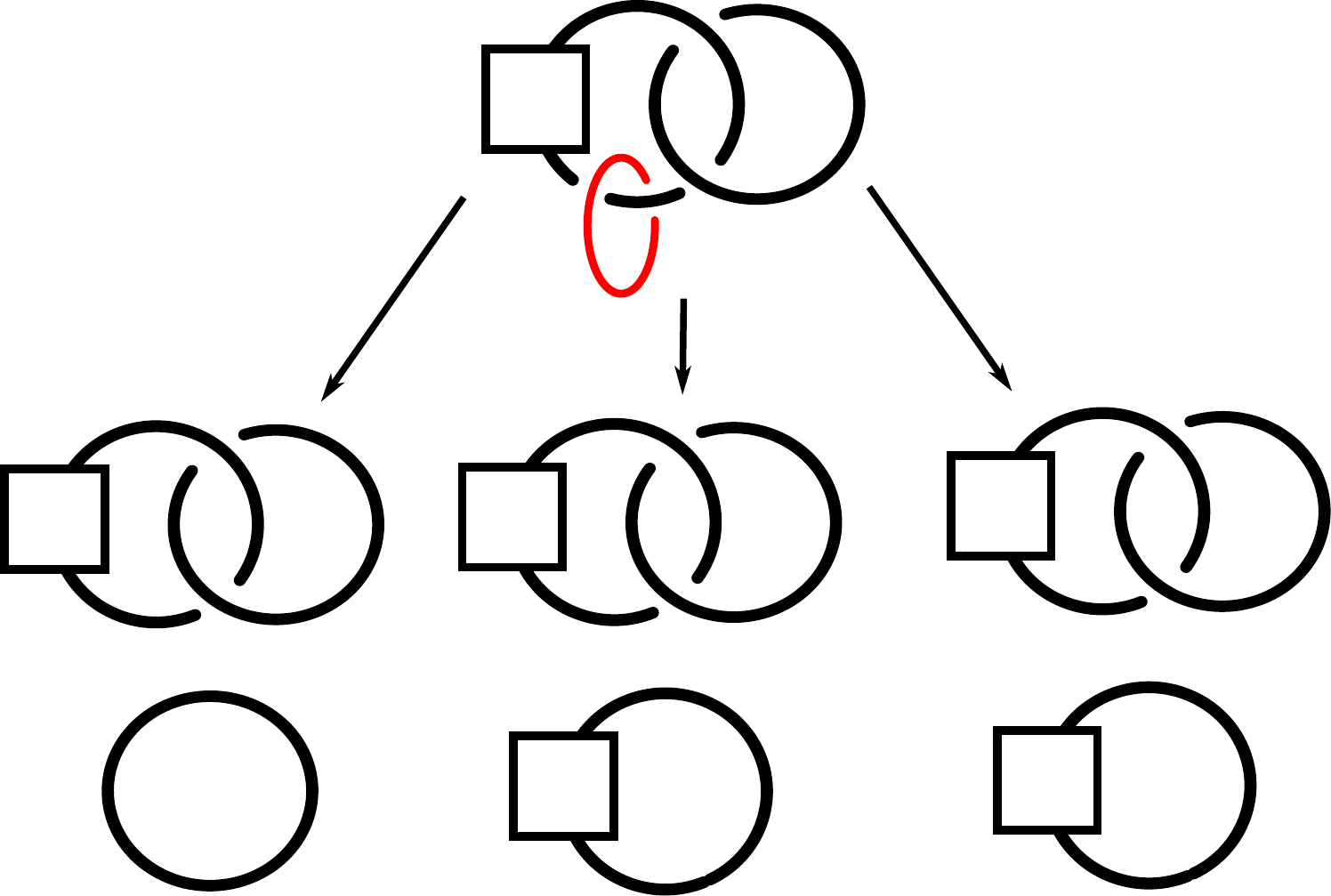}
      \put (2.5,27) {\Large $K$}
      \put (37,27) {\Large $K$}
      \put (73.5,27.5) {\Large $K$}
      \put (40,7) {\Large $K$}
      \put (77,7) {\Large $K$}
      \put (38,58) {\Large $K$}
      \put (12,16.5) {\Large $\cong$}
      \put (46,16.5) {\Large $\cong$}
      \put (84,16.5) {\Large $\cong$}
      \put (8,-3) {\Large $Y_0\cong U(q/t)$}
      \put (40,-3) {\Large $Y_1\cong K(p/q)$}
      \put (72,-3) {\Large $Y\cong K((p+q)/q)$}
      \put (3,35) {$\frac{1}{0}$}
      \put (38,34) {$n$}
      \put (77,37) {$n+1$}
      \put (60,11) {$\frac{p}{q}$}
      \put (94,11) {$\frac{p+q}{q}$}
      \put (40,66) {$n$}
      \put (28,34) {$\frac{q}{t}$}
      \put (63,34) {$\frac{q}{t}$}
      \put (98,37) {$\frac{q}{t}$}
      \put (65,66) {$\frac{q}{t}$}
      \put (25,11) {$\frac{q}{t}$}
      \put (40,48) {\Large $C$}
    \end{overpic}
  }
  \vspace{0.5cm}
\caption{The triad obtained by performing $0$, $1/0$ and $-1$ surgery on the knot $C$}
\label{fig:triad1}
\end{figure}
When paired with the existence of quasi-alternating surgeries on asymmetric knots, this necessarily implies the existence of many formal L-spaces that are not obtained as the double branched cover of any link in $S^3$.
\begin{cor}\label{cor:asymmetric_formal_L_spaces}
    There exist infinitely many asymmetric formal L-spaces. In particular, there are infinitely many formal L-spaces that are not the double branched cover of any link in $S^3$.
\end{cor}

\begin{proof}
Let $K$ be any asymmetric knot admitting an alternating or quasi-alternating surgery. For example, let $K$ be one of the asymmetric L-space knots in the SnapPy census or one of the Baker--Luecke asymmetric L-space knots. By Corollary~\ref{cor:computableexceptionalsymmetryslopes}, for all $r$ sufficiently large $K(r)$ is an asymmetric L-space. However, infinitely many of these surgeries are also formal L-spaces by Theorem~\ref{thm:plusone}.
\end{proof}
For integer surgeries yielding a formal L-space, Theorem~\ref{thm:plusone} can be strengthened.
\begin{thm}\label{thm:formal_integer}
    If $K(n)$ is a formal L-space for some $n\in \Z_{>0}$, then $K(r)$ is a formal L-space for every $r\geq n$.
\end{thm}

\begin{proof}
Suppose that $K(n)$ is a formal L-space for some $n\in \Z_{>0}$. For simplicity, we may assume that $n\geq 2$.  
Recall that every rational number $p/q > 1$ has a unique negative continued fraction expansion of the form
\[
p/q=[a_1,\dots, a_\ell]^-,
\]
where $a_i\geq 2$ for all $i$. We call $\ell$ the length of this continued fraction expansion. Note also that if $p/q>n$, then $a_1>n$. We prove the theorem by induction on $\ell$ and then by induction on the final coefficient $a_\ell$. The base case for this induction is the case $\ell=1$. In this case, $p/q$ is an integer and Theorem~\ref{thm:plusone} gives the necessary result.

Thus suppose that $p/q>n$ has continued fraction expansion
\[
p/q=[a_1,\dots, a_{\ell -1}, a_\ell]^-,
\]
where $a_i\geq 2$ for all $i$ and $\ell>1$. Let $p_0/q_0$ be the rational number
\[
p_0/q_0=[a_1,\dots, a_{\ell -2},a_{\ell -1}]^-.
\]
Since $a_i\geq 2$ for all $i$, we have that $p_0/q_0>a_1-1\geq n$ and the length of this continued fraction is $\ell-1$. Thus, we may inductively suppose that $K(p_0/q_0)$ is a formal L-space.
Let $p_0/q_0$ be the rational number
\[
p_1/q_1=[a_1,\dots, a_{\ell -1},a_{\ell}-1]^-.
\]
Note that $p_1/q_1$ satisfies
\[
p_1/q_1=[a_1,\dots,a_{k-1}, a_{k}-1]^-,
\]
where $k\leq \ell$ is maximal such that $a_k>2$.
In particular, $p_1/q_1\geq a_1-1\geq n$ and by induction either on the length or final coefficient of the continued fraction expansion we may assume that $K(p_1/q_1)$ is a formal L-space.

On the other hand, Figure~\ref{fig:triad2} shows that the manifolds $K(p_0/q_0)$, $K(p_1/q_1)$ and $K(p/q)$ sit in a surgery triad. Furthermore standard continued fraction identities show that $\frac{p}{q}=\frac{p_0 + p_1}{q_1+q_0}$. Thus we see that the homology groups of these manifolds satisfy the relation from the definition of a formal L-space. Since $K(p_0/q_0)$ and $K(p_1/q_1)$ are formal L-spaces, it follows that $K(p/q)$ is also a formal L-space. 
\end{proof}

The fact that a Seifert fibered space is a formal L-space if and only if it is the double branched cover of a quasi alternating link \cite{Is18} implies that for torus knots the set of formal L-space slopes coincides exactly with the set of quasi-alternating slopes.
\begin{cor}\label{cor:formal_torus}
    For positive torus knots the set of formal L-space surgeries agrees with the set of quasi-alternating surgeries given in Theorem~\ref{thm:torus_knotsQA}.
\end{cor}

\begin{proof}
     Surgery on a torus knot is a Seifert fibered space over $S^2$ (including the possibility of a lens space) or a connected sum of two lens spaces. From Theorem~\ref{thm:Moser} it follows that $T_{a,b}(p/q)$ is a connected sum of two lens spaces if and only if $p/q=ab$ which is by Theorem~\ref{thm:torus_knotsQA} a quasi-alternating slope and thus also a formal L-space slope. 
     In his classification of quasi-alternating Montesinos knots, Issa showed that a Seifert fibered space over $S^2$ is a formal L-space if and only if it is the double branched cover of a quasi-alternating link \cite{Is18}. 
\end{proof}
Finally we check that every non-trivial L-space knot has L-space surgery slopes that are not formal L-space surgeries.
\begin{prop}\label{prop:formal_bound}
    Let $K$ be a non-trivial knot and $r>0$ be a formal L-space surgery. Then
    \[ r > 2g(K) + \frac{1}{2}\left(\sqrt{1+8g(K)}-1\right).\]
\end{prop}

\begin{proof}
    Like the double branched cover of a quasi-alternating link, a formal L-space is the boundary of a negative-definite manifold with $H_1=0$. If $Y_0$ and $Y_1$ are a pair of 3-manifolds appearing in a surgery triad for $Y$ as in the definition of a formal L-space, then it can be verified that one of the surgery cobordisms from $Y_0$ or $Y_1$ to $Y$ will be negative definite (see \cite[Section~2.2]{Lidman2017}, for example). This allows one to build a negative-definite manifold inductively via 2-handle attachments to $B^4$. However, if $K(r)$ bounds a negative definite manifold with $H_1=0$, then $K(\lceil r \rceil)$ also bounds such a manifold, since there is a negative-definite 2-handle cobordism from $K(r)$ to $K(\lceil r \rceil)$ \cite[Lemma~2.6]{Owens_Strle12}. The work of Greene \cite{Greene2015genusbounds} implies that we then have the bound 
    \[2g(K)\leq \lceil r \rceil - \sqrt{\lceil r \rceil}\]
    which is equivalent to
    \[(\sqrt{2g(K)+1/4}+1/2)^2\leq \lceil r \rceil .\]
    Thus we obtain
    \[r>\lceil r \rceil-1\geq 2g(K)-\frac{1}{2}+\sqrt{2g(K)+1/4},\]
    which is equivalent to the claimed inequality.  
\end{proof}

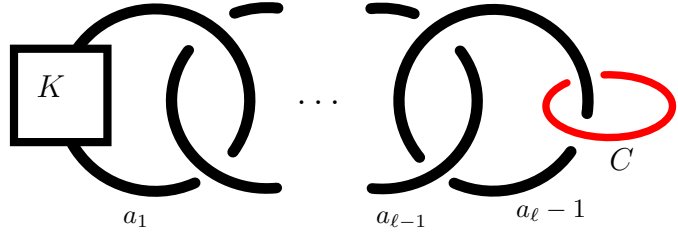
\begin{figure}%[!ht]
  \centerline{
    \begin{overpic}[width=0.7\textwidth]{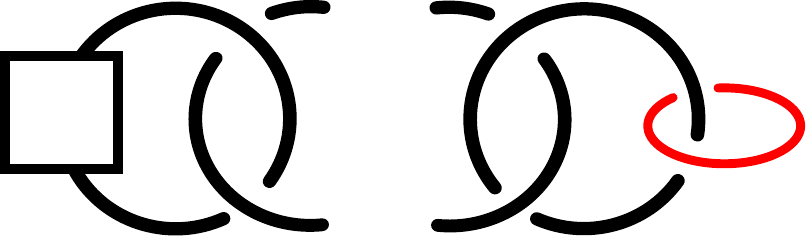}
      \put (4,15) {\Large $K$}
      \put (43,14) {\LARGE $\dots$}
      \put (90,4) {\Large $C$}
      \put (17,-4) {$a_1$}
      \put (76,-3) {$a_\ell-1$}
      \put (55,-4) {$a_{\ell-1}$}
    \end{overpic}
  }
  \vspace{0.5cm}
\caption{The triad obtained by performing $0$, $1/0$ and $-1$ surgery respectively on the knot $C$ yields $K(p_0/q_0)$, $K(p_1/q_1)$ and $K(p/q)$.}
\label{fig:triad2}
\end{figure}

\section{Observations and questions}\label{sec:questions}
	
We finish this article by stating some observations and questions.

 \medskip
The strategies from Section~\ref{sec:strategies} together with the main result of~\cite{Twisting_QA} and~\cite{Watson_KH_sym2017} work well in practice to determine the quasi-alternating/thin status for large infinite subsets of slopes for a given knot. More concretely: For an asymmetric knot, every quasi-alternating/thin slope has to be one of the finitely many symmetry-exceptional-symmetric slopes, and thus one can evoke Strategies~\ref{strat:symmetryexceptional_slopes}, \ref{strat:branchlocussearch}, and \ref{strat:singleQA_slope}. If an L-space knot $K$ is not asymmetric, then it is conjectured that it admits a single involution~\cite{L_space_periodic_conj}. In that case, we can use Strategy~\ref{strat:dbc_slopes} to get all branching sets as tangle fillings $T(r)$ of the tangle exterior $T$. The main result of~\cite{Twisting_QA} gives a criterion to demonstrate that $T(r)$ is quasi-alternating for all $r$ in an interval of the form $[a,\infty)$, while~\cite{Watson_KH_sym2017} can be used to show that infinitely $T(r)$ is Khovanov thick for all $r$ in an interval of the form $(-\infty,b)$. However, in general, $a$ and $b$ do not coincide.

\begin{ques}
	Does there exist an algorithm to find the set of quasi-alternating/thin slopes for a given knot $K$ in $S^3$?
\end{ques}

The classification of such slopes for the census knots would be one useful application of such an algorithm, or at least a testing ground for designing one.
\begin{ques}
	What is the classification of the quasi-alternating/thin slopes of knots in the SnapPy census?
\end{ques}

To highlight the complexity of this problem, it even seems to be unknown which slopes are thin for torus knots.
	
\begin{ex}\label{ex:torus_thin}
	Consider the torus knot $T_{5,3}$. Theorem~\ref{thm:torus_knotsQA} shows that the quasi-alternating slopes are $(12+1/2,\infty)$. In Example~\ref{ex:torus} we have observed that the $(12+1/2)$-surgery on $T_{5,3}$ is the double branched cover of the knot $K11n50$ which is thin but not quasi-alternating~\cite{Greene_thin_QA}. 
	If we apply \cite{Watson_KH_sym2017} we get that all slopes smaller or equal to $12$ are thick. Thus the slopes in the interval $(12,12+1/2)$ remain of undetermined thick/thin status. However, experiments suggest that all surgeries in that interval are thin. For example: $K10n21(12+1/3)=DBC(K12n296)$ and $K10n21(12+1/4)=DBC(K13n2006)$ and both branching sets are thin but not quasi-alternating. 
\end{ex}
	
\begin{ques}\label{ques:thin_torus}
	What is the classification of the thin slopes of torus knots?
\end{ques}

Most surgeries on torus knots are small Seifert fibered spaces \cite{Moser1971elementary} and these arise as double branched covers of Montesinos links and Seifert links. Since the Seifert links with L-space double branched covers are also Montesinos links \cite{ADE_link}, this naturally leads to the question of thinness for Montesinos links.

\begin{ques}
	What is the classification of thin Montesinos links?
\end{ques}
 
Computing the Khovanov homology of a link is algorithmic (though perhaps not practical), and hence so is determining whether any given link is thin. 
\begin{ques}\label{ques:QA_decidable}
	Is there an algorithm to detect if a given link $L$ in $S^3$ is quasi-alternating or not?
\end{ques}

Note that from the inductive definition of quasi-alternating links, it follows that there exists an algorithm that certifies in finite time if a knot is quasi-alternating. Thus Question~\ref{ques:QA_decidable} reduces to the existence question for an algorithm that certifies if link is not quasi-alternating.

\medskip
	
	In Theorem~\ref{thm:asym_QA} we observed that every asymmetric census L-space knot has exactly two dbc slopes, these two dbc slopes
	are consecutive integers, and each of these dbc slopes is quasi-alternating.
	Theorem~\ref{thm:BLknots} shows that a handful of the Baker--Luecke asymmetric L-space knots behave similarly (where the two quasi-alternating slopes are actually both alternating).
	Thus we wonder the following, though are skeptical that any are affirmative in general. 
\begin{ques}\label{ques:asym}\phantom{xxxxxx}
	\begin{enumerate}
		\item Does every hyperbolic asymmetric L-space knot have exactly two dbc slopes?  
		\item Suppose a hyperbolic asymmetric L-space knot has exactly two dbc slopes: 
		\begin{enumerate}
			\item Are they consecutive integers?
			\item Are they both quasi-alternating?
			\item Do the branching sets for the surgered manifolds have the same crossing number?
		\end{enumerate} 
	\end{enumerate}
\end{ques}
	
\let\MRhref\undefined
\bibliographystyle{hamsalpha}
\bibliography{altsurg.bib}

\end{document}